\documentclass{daj}
\usepackage{graphicx}

\usepackage{amsmath}
\usepackage{amssymb}
\usepackage{amsfonts}
\usepackage{amsthm}
\usepackage{mathrsfs}
\usepackage{enumerate}
\usepackage{float}

\newtheorem{theorem}{Theorem}[section]
\newtheorem{proposition}[theorem]{Proposition}
\newtheorem{lemma}[theorem]{Lemma}
\newtheorem{definition}[theorem]{Definition}
\newtheorem{corollary}[theorem]{Corollary}

\newtheorem{remark}{Remark}[section]
\newtheorem{example}[theorem]{Example}

\numberwithin{equation}{section}

\newcommand{\rr}{{\mathbb R}}
\newcommand{\zz}{{\mathbb Z}}
\newcommand{\nn}{{\mathbb N}}

\newcommand{\qq}{{\mathbb Q}}

\newcommand{\RR}{{\mathbb R}}
\newcommand{\ZZ}{{\mathbb Z}}
\newcommand{\NN}{{\mathbb N}}

\newcommand{\calS}{\mathcal{S}}

\newcommand{\dist}{\hbox{dist\,}}
\newcommand{\Fav}{\hbox{\rm Fav}}

\newcommand{\jfrak}{{\mathfrak J}}

\newcommand{\C}{\mathbb{C}}

\newcommand{\bbA}{\mathbb{A}}

\newcommand{\calm}{\mathcal{M}}

\newcommand{\calC}{\mathcal{C}}

\dajAUTHORdetails{%
  title = {Vanishing Sums of Roots of Unity and the Favard Length of Self-Similar Product Sets}, 
  author = {Izabella {\L}aba and Caleb Marshall},
  plaintextauthor = {Izabella {\L}aba, Caleb Marshall},
}   

\dajEDITORdetails{%
   year={2022},
   number={19},
   received={17 March 2022},   
   published={16 December 2022},  
   doi={10.19086/da.57602},       
}   

\begin{document}

\begin{frontmatter}[classification=text]


\author[ilaba]{Izabella {\L}aba\thanks{Supported by NSERC Discovery Grant 22R80520}}
\author[cmarshall]{Caleb Marshall\thanks{Supported by NSERC Discovery Grants 22R80520 and GR010263}}

\begin{abstract}
We improve a special case of the Lam-Leung lower bound \cite{LL} on the number of elements in a vanishing sum of $N$-th roots of unity.  Using this result, we extend the Favard length estimates due to Bond, {\L}aba, and Volberg \cite{BLV} to a new class of rational product Cantor sets in $\mathbb{R}^2$.
\end{abstract}
\end{frontmatter}

\section{Introduction}
A {\em vanishing sum of roots of unity} is an expression of the form
\begin{equation}\label{vanishing}
z_1+\dots+z_k=0,
\end{equation}
where $z_1,\dots,z_k$ are $N$-th roots of unity for some $N\in\NN$.  In general, we do not require the $z_j$ to be distinct or primitive $N$-th roots of unity.
Vanishing sums of roots of unity have been studied extensively in number theory, see e.g., \cite{CDK}, \cite{CJ}, \cite{CS}, \cite{deB}, \cite{LL}, \cite{Mann}, \cite{PR}, \cite{Re1}, \cite{Re2}, \cite{schoen}.  Of particular interest is the following result, taken from \cite{LL}.

\begin{theorem}\label{LL-thm}
Let $N=\prod_{j=1}^K p_j^{\alpha_j}$, where $p_1,\dots,p_K$ are distinct primes and $\alpha_1,\dots,\alpha_K\in\NN$. Suppose that (\ref{vanishing}) holds, where $z_1,\dots,z_k$ are $N$-th roots of unity. 
Then there exist nonnegative integers $a_1,\dots,a_K$ such that
$k=\sum_{j=1}^K a_jp_j$. In particular, $k\geq \min \{p_j\}$.
\end{theorem}

A key principle motivating our work is that (\ref{vanishing}) can be rephrased in terms of divisibility of polynomials by cyclotomic factors. 
Recall that the $s$-\textbf{th cyclotomic polynomial} $\Phi_s (X)$ is the unique monic, irreducible polynomial whose roots are the primitive $s$-th roots of unity.
In other words, $z\in\mathbb{C}$ is a root of $\Phi_s(z)=0$ if and only if we have $z=e^{2\pi i d/s}$ for some $d\in\ZZ$ with $(d,s)=1$. 
Alternatively, $\Phi_s$ can be defined inductively via the identity
\begin{equation}\label{poly-e0}
X^N-1=\prod_{s|N}\Phi_s(X).
\end{equation}
Thus $\Phi_1(X)=X-1$, and 
\begin{equation}\label{polysec-e0}
1+X+X^2+\dots+X^{N-1}=\prod_{s|N,s\neq 1}\Phi_s(X)\hbox{ for }N>1.
\end{equation} 
In particular, if $p$ is a prime number, then $\Phi_p(X)=1+X+\dots+X^{p-1}$.
Thus
$$
0=\Phi_p(e^{2\pi i/p})=1+e^{2\pi i /p} + \dots + e^{2\pi i (p-1)/p}
$$
provides an example of a vanishing sum of roots of unity. In general, however, such sums can be much more complicated.

Assuming that $z_1,\dots,z_k$ are (not necessarily primitive) $N$-th roots of unity,  we have $z_\ell=e^{2\pi i a_\ell/N}$ for some $a_\ell\in\ZZ$. 
Then (\ref{vanishing}) holds if and only if 
\begin{equation}\label{vanishing2}
A(e^{2\pi i /N})=0,
\end{equation}
where
$A(X)=\sum_{a=0}^{N-1} w(a) X^a$ and
$$
w(a) =|\{ \ell\in\{1,\dots,k\} : z_\ell = e^{2 \pi i a_\ell /N} \hbox{ and } \ a_\ell \equiv a\hbox{ mod }N\}|.
$$
Since $\Phi_N(X)$ is the minimal polynomial of $e^{2\pi i/N}$, it follows that (\ref{vanishing2}) holds if and only if $\Phi_N(X)|A(X)$. 

The lower bound in Theorem \ref{LL-thm} now takes the following form: suppose $\Phi_N(X)|A(X)$, where $A(X)$ is a polynomial with non-negative integer coefficients as above. Then 
\begin{equation}\label{LL-bound}
A(1)=\sum_{a=0}^{N-1} w(a)\geq \min_j p_j,
\end{equation}
where $\, p_1,\dots,p_K$ are the distinct prime factors of $N$.

We are interested in improvements to (\ref{LL-bound}) when $A(X)$ is assumed to have multiple cyclotomic divisors. In general, such improvements are not possible without some additional assumptions on the cyclotomic factors of $A(X)$. To demonstrate, consider the simple example
\begin{equation}\label{rescaled-fiber}
A(X)=1+X^{q^j}+X^{2q^j}+\dots +X^{(p-1)q^j}= \Phi_p(X^{q^j}),
\end{equation}
where $j\in\NN$ and $p<q$ are distinct primes. For each $\alpha\in\{0,1,\dots,j\}$, the number $e^{2\pi i q^{j-\alpha}/p}$ is a $p$-th primitive root of unity, therefore a root of $\Phi_p (X)$. Thus $A(e^{2\pi i/pq^\alpha })=0$, and consequently $\Phi_{pq^\alpha}(X)|A(X)$, for all $\alpha\in\{0,1,\dots,j\}$; yet, we still only have $A(1)=p$.  Since $j \in \NN$ is arbitrary, we can find polynomials $A(X) $ such that $A(1)$ attains the lower bound of Theorem \ref{LL-thm}, but $A(X)$ has as many cyclotomic factors as we like.

However, we do have an improvement in the following case.

\begin{theorem}\label{main theorem}
Let $A(X)$ be a polynomial with non-negative coefficients and distinct cyclotomic factors $\Phi_{m_1}(X), ..., \Phi_{m_k}(X)$.  
Assume that there exist distinct prime numbers $p,q$, and exponents $\alpha_j, \beta_j \in \NN\cup\{0\}$ such that $m_j = p^{\alpha_j} q^{\beta_j}$ for each $1 \leq j \leq k$.  Assume further that $q \nmid A(1)$.  Then we have the lower bound
$$
A(1) \geq p^{E_p}, \hbox{ where } E_p = |\{\alpha_1, ..., \alpha_k\}|.
$$
In words, $E_p$ denotes the number of \textbf{distinct} exponents $\alpha_i$ appearing among the $m_j = p^{\alpha_j} q^{\beta_j}$.
\end{theorem}

A similar statement holds with $p$ and $q$ interchanged, so that if we assume that $p \nmid A(1)$ instead of $q \nmid A(1)$, we have an analogous lower bound in terms of $q$ and its exponents. If neither of $p,q$ divides $A(1)$, we may choose the maximum of the two lower bounds.

\begin{remark}
The conclusion of Theorem \ref{main theorem} is immediate in the special case when $m_j = p^{\alpha_j}$ are all powers of the same prime $p$ with $\alpha_1,\dots,\alpha_k\geq 1$.
 Indeed, it is easy to deduce by induction from (\ref{polysec-e0}) that if $p$ is prime and $\alpha\in\NN$, then
\begin{equation}\label{prime-power-cyclo}
\Phi_{p^\alpha}(X)=\Phi_p(X^{p^{\alpha-1}}) = 1 + X^{p^{\alpha-1}} + X^{2p^{\alpha-1}} + \dots + X^{(p-1)p^{\alpha-1}},
\ \ \alpha\geq 1.
\end{equation}
Hence 
\begin{equation}\label{prime-power-cyclo-2}
\Phi_{p^{\alpha_j}}(1)=p,
\end{equation} 
and $\prod_{j=1}^k\Phi_{p^{\alpha_j}}(X)|A(X)$ implies that 
$p^k=\prod_{j=1}^k\Phi_{p^{\alpha_j}}(1)|A(1)$. In particular, $p^k\leq A(1)$, as claimed. 

The same argument does not apply to cyclotomic factors of the form $\Phi_{ p^{\alpha_j} q^{\beta_j}}$ with $\alpha_j,\beta_j\geq 1$, or, more generally, to $\Phi_s$ with $s$ composite. In fact, (\ref{polysec-e0}) and (\ref{prime-power-cyclo-2}) imply that for all $N>1$,
$$
N=\prod_{s|N,s\neq 1}\Phi_s(1)
= N \prod_{s|N,\ s\text{ is not prime}}\Phi_s(1),
$$
so that $\Phi_s(1)=1$ for all $s>1$ that are not prime. This is why, in examples such as (\ref{rescaled-fiber}), we can get $A(X)$ to have an arbitrarily large number of cyclotomic factors $\Phi_s$ with composite $s$ while keeping $A(1)$ constant.
\end{remark}

Theorem \ref{main theorem} is motivated in part by its application to the Favard length problem for rational product Cantor sets.  We give a brief introduction to this problem now, and state the relevant previous results, before presenting the extension derived from our Theorem \ref{main theorem}.

Let $A,B\subset\NN$ be finite sets with $\min(|A|,|B|)\geq 2$, and let $L:=|A||B|$. For $n\in\NN$, define the sets $A_n$ and $B_n$ recursively via the formulas $A_1:=L^{-1}A$, $A_{n+1}:=A_n+L^{-n-1}A$, and similarly for $B_n$. 
Let
\begin{equation}\label{def-S}
\calS_n=A_n\times B_n+\{x\in\RR^2 :|x|<L^{-n}\}, \ \ n\in\NN.
\end{equation}
We may think of $\calS_n$  as the $n$-th iteration of a self-similar set $\calS_\infty\subset\RR^2$, defined as follows. 
Let $\{z_j\}_{j=1}^L =A\times B\subset\RR^2$; since $|A|, |B| \geq 2$, these points are distinct and non-collinear.
Let $T_1,\dots,T_L:\C\to\C$ be the similarity maps $T_j(z)=\frac1{L}z+z_j$. We then define $\calS_\infty$ to be the unique compact set such that $\calS_\infty=\bigcup_{j=1}^L T_j(\calS_\infty)$. 
It is well known that such a set exists, has Hausdorff dimension at most 1 (equality follows if the \textbf{open set condition} holds, see  \cite{mattila}) and finite 1-dimensional Hausdorff measure.

For $\theta\in[0,\pi]$, let $proj_\theta:\RR^2\to\RR$ be the linear projection defined by $proj_\theta(x)=x_1\cos\theta+ x_2\sin\theta$ for $x=(x_1,x_2)\in\RR^2$.
Since the $z_j$ are not collinear, $\calS_\infty$ is unrectifiable, and it follows from a theorem of Besicovitch that $ |proj_\theta(\calS_\infty)|=0$ for almost every $\theta \in [0, \pi]$ (see \cite{mattila}).  In particular, if we define the  \textbf{Favard length} of a planar set $\calS$ as the average (with respect to angle) length of its projections,
$$
Fav(\calS):=  \frac{1}{\pi} \int_0^{\pi} |proj_{\theta} (\calS) | ,
$$
then
\begin{equation}\label{fav-limit}
\lim_{n\to\infty}Fav(\calS'_n)=Fav(\calS_\infty)=0,
\end{equation}
where $\calS'_n$ is the $L^{-n}$-neighborhood of $\calS_\infty$.  
While $\calS'_n$ is not necessarily equal to the finite scale set $\calS_n$ defined above, it can be covered by a bounded (independently of $n$) number of copies of $\calS_n$, and vice versa. Therefore any quantitative estimates on $Fav(\calS_n)$ and $Fav(\calS'_n)$ are equivalent up to constants.

Although it is immediate that $Fav (\calS_n) \rightarrow 0$ as $n \rightarrow \infty$, a more subtle question concerns estimating the rate of decay, both from above and from below. There has been significant interest in this issue, with Favard length estimates proved for various types of fractal sets in \cite{BV}, \cite{BThesis}, \cite{BLV}, \cite{BV1}, \cite{BV3}, \cite{Bongers}, \cite{CDOV}, \cite{LZ}, \cite{NPV}, \cite{Tao}, \cite{VV}, \cite{Wilson}, \cite{Zhang}. Motivation and an overview is provided in the review articles \cite{L1}, \cite{PS1}, and applications to analytic capacity and Vitushkin's conjecture are discussed in \cite{DV}, \cite{To}; see also  \cite{BV2}, \cite{BT}, \cite{CDT} for results of this type for curvilinear projections.

The following theorem, proved by Bond, Volberg, and the first author in \cite{BLV}, gives an asymptotic estimate on $Fav(\calS_n)$ when $A$ and $B$ have small cardinality.

\begin{theorem}\label{BLV-thm}
Define $\calS_n$ as in (\ref{def-S}). If $|A|,|B|\leq 6$, then $Fav(\calS_n)\lesssim n^{-\epsilon/\log\log n}$ for some $\epsilon>0$.
\end{theorem}
Here and below, we use the $\lesssim$ notation as follows. Let $G(n)$ be a function defined for all sufficiently large $n\in\NN$, with values in $(0,\infty)$. Then $Fav(\calS_n) \lesssim G(n)$ means that there exists a constant $C > 0$, independent of $n$, such that $Fav (\calS_n) \leq C G(n)$ for all $n$ large enough.

The proof in \cite{BLV} relies upon the cyclotomic divisibility properties of certain polynomials associated to $A$ and $B$.  Specifically, define
\begin{equation}\label{generating}
A(X)=\sum_{a\in A} X^a,
\end{equation}
and similarly for $B$.  We are assuming that $A,B\subset\NN$, so that $A(X)$ and $B(X)$ are polynomials in $\zz[X]$. For our present purposes, it will suffice to consider $A$ and $B$ separately; and so, we present these requirements for the polynomial $A(X)$.  Recalling that $L = |A||B|$, we define the following factors of $A(X)$.

\begin{definition}\label{A1234}
We have $A(X)=\prod_{i=1}^4 A^{(i)}(X)$, where each $A^{(i)}(X)$ is a product of  irreducible factors of $A(X)$ in $\zz[X]$, chosen as follows (by convention, an empty product is identically equal to 1):
\begin{itemize}
\item $A^{(1)}(X)=\prod_{s\in S_A^{(1)}}\Phi_s(X)$, $S_A^{(1)}=\{s\in\nn:\ \Phi_s(X)|A(X), (s,L)\neq 1\}$,
\item $A^{(2)}(X)=\prod_{s\in S_A^{(2)}}\Phi_s(X)$, $S_A^{(2)}=\{s\in\nn:\ \Phi_s(X)|A(X),$ $(s,L)=1\}$,
\item $A^{(3)}(X)$ is the product of those irreducible factors of $A(X)$ that have at least one root of the form $e^{2\pi i\xi_0}$, $\xi_0\in\rr\setminus\qq$,
\item $A^{(4)}(X)$ has no roots on the unit circle.
\end{itemize}
We then define the \textbf{good} and \textbf{bad} factors $A'$ and $A''$ of $A$:
\begin{equation}\label{a'a''}
A'(X):=A^{(1)}(X)A^{(3)}(X)A^{(4)}(X),\ \ A''(X):=A^{(2)}(X).
\end{equation}
\end{definition}

Theorem \ref{BLV-thm} is then a consequence of the following more general result. We retain the notation of \cite{BLV}  for the statement of the following theorem. Afterword, we simplify this notation to better state our main theorem.

\begin{proposition}\label{single prime} \cite{BLV}
Let $A,B$ be as in Theorem \ref{BLV-thm}, but without the assumption that $\max(|A|,|B|)\leq 6$.  As before, define
\begin{equation}\label{def-SA}
S_A^{(2)}:= \{s:\Phi_s|A\text{ and }(s,L)=1\}= \{s:\Phi_s(X)|A''(X)\},\ \ 
s_A:=\text{lcm}(S_A^{(2)}).
\end{equation}
Suppose we can factor $s_A=s_{1,A}s_{2,A}$, with $s_{1,A},s_{2,A}\in\NN$ such that:
\begin{itemize}
\item  $s_{2,A}<|A|$,
\item $\Phi_s (X)$ does not divide $A(X)$ for any $s|s_{1,A}$.
\end{itemize}
Assume also that the same holds for $B$. Then $Fav(\calS_n)\lesssim n^{-\epsilon/\log\log n}$ for some $\epsilon>0$. 
\end{proposition}

It is verified in \cite{BLV} that the assumptions of Proposition \ref{single prime} must hold if $\max(|A|,|B|)\leq 6$, hence Theorem \ref{BLV-thm} follows. The same proof also yields a stronger result in the case when $A^{(3)}=B^{(3)}=1$. 

\begin{theorem}\label{thm-noliouville} \cite{BLV}
Let $A,B$ be as in Theorem \ref{BLV-thm} or Proposition \ref{single prime}. Assume that all roots of $A(X)$ and $B(X)$ on the unit circle are roots of unity. Then
$Fav(\calS_n)\lesssim n^{-\epsilon}$ for some $\epsilon>0$.
\end{theorem}

For sets $A,B$ not satisfying the assumptions of Proposition \ref{single prime},  the only remaining barrier in for proving upper estimates on $\Fav(\calS_n)$ is an analysis of the bad factor $A^{''}$ associated to the set $S_A^{(2)}$.  As such, we restrict our attention to this collection of ``bad'' cyclotomic factors.  This allows us to simplify notation and write $S_A : = S_A^{(2)}$. 

Combining the framework of \cite{BLV} with our Theorem \ref{main theorem} leads to the following result.

\begin{theorem}\label{Favard-two-primes}
Define $\calS_n$ as in (\ref{def-S}).
Let $A(X) = \sum_{a \in A} X^a$, and similarly for $B(X)$.  Define $S_A : = S_A^{(2)}$ and $s_A$ as in (\ref{def-SA}), and similarly for $B$. Assume that $s_A$ and $s_B$ have at most two prime divisors each (not necessarily the same ones). Then
$$
Fav (\calS_n) \lesssim n^{- \epsilon / \log \log n}
$$
for some $\epsilon>0$. If, in addition, all roots of $A(X)$ and $B(X)$ on the unit circle are roots of unity,  we have the improvement $Fav(\calS_n) \lesssim n^{- \epsilon}$ for some $\epsilon > 0$.
\end{theorem}

\begin{remark}
The assumptions of Theorem \ref{Favard-two-primes} {\it do not} imply that the sets $A,B$ satisfy the assumptions of Proposition \ref{single prime}; see Section \ref{example1} for an example.
\end{remark}

\begin{remark}\label{remark1}
It is not difficult to see that elements of $S_A$ cannot be prime powers. Indeed, by (\ref{prime-power-cyclo}) we have $\Phi_{p^\alpha}(1)=p$ if $p$ is prime and $\alpha\in\NN$. 
Hence $\Phi_{p^\alpha}(X)|A(X)$ implies that $p|A(1)=|A|$, and in particular that $p^\alpha$ cannot be relatively prime to $|A|$. 

It follows that if $S_A$ is nonempty, then $s_A$ must have \textit{at least} two distinct prime factors. Thus Theorem \ref{Favard-two-primes} represents the simplest rational product case that goes beyond \cite{BLV}.
\end{remark}

We are also able to increase the cardinality bound in Theorem \ref{BLV-thm}.

\begin{theorem}\label{favard-small}
Theorem \ref{BLV-thm} and Theorem \ref{thm-noliouville} continue to hold with the assumption that $\max(|A|,|B|)\leq 6$ replaced by $\max(|A|,|B|)\leq 10$.
\end{theorem}

The cardinality bound $10$ in Theorem \ref{favard-small} is not a ``hard" one, in the sense that the arguments used in the proof of the theorem continue to work for many larger sets. However, the number of cases to consider increases significantly for sets of cardinality 11 and higher, making the proof more difficult to manage without contributing new ideas.

This article is organized as follows. The proof of Theorem \ref{main theorem} is given in Section \ref{main-proof}, with the notation and basic cyclotomic divisibility tools provided in Section \ref{cyclo-tools}. This part of the paper is 
self-contained and does not involve any Favard length considerations, 
so that the reader interested only in Theorem \ref{main theorem} can work with just these two sections.

The proof of Theorem \ref{Favard-two-primes} consists of several steps. 
In Section \ref{sec2}, we explain how to combine the arguments of \cite{BLV} (specifically, the {\it Set of Large Values} approach) with our Theorem \ref{main theorem} and our main intermediate result, Proposition \ref{lb-general-gamma}, to prove the theorem. The proof of Proposition \ref{lb-general-gamma}, given in Section \ref{SLV-section}, contains the main new ideas of this paper with regard to Favard length estimates. Specifically, while we follow the general approach of \cite{BLV}, we are able to use it more efficiently by splitting up the cyclotomic divisors of $A(X)$ and $B(X)$ into appropriate ``clusters'' and reorganizing the Set of Large Values construction accordingly. 

An important part of the proof of Theorem \ref{Favard-two-primes} is that the lower bounds on $A(1)$ depending on the cyclotomic divisors of $A(X)$ (such as those in Theorem \ref{main theorem}) have to match the size of a Set of Large Values that we can identify. 
In Section \ref{sec-examples}, we provide examples illustrating this. We also discuss briefly the barriers to extending our results to the case when $s_A$ or $s_B$ have 3 or more prime factors. Essentially, while certain simple examples of this type can be handled using the cuboid arguments in Section \ref{cuboid-lower-bound}, a systematic approach to the general case would require additional new ideas.
We conclude the paper with the proof of Theorem \ref{favard-small} in Section \ref{small-card}.


\section{Proof of Theorem \ref{Favard-two-primes}}\label{sec2}

Our proof follows the general approach of \cite{BLV}, but with an additional argument to allow an application of Theorem \ref{main theorem}. Most of the proof in \cite{BLV} applies to general rational product sets; the only part that requires either the restriction  
$\max(|A|,|B|)\leq 6$ or the cyclotomic divisibility assumption in Proposition \ref{single prime} is the SLV (\textbf{Set of Large Values}) argument in Sections 5 and 6. Below, we provide a short summary of what is needed, then explain how to modify this step in our two-prime setting. The proofs of the intermediate results are postponed until later sections.

We first state the SLV result we need.
Define $\phi_A^{''} (\xi) = A''(e^{2 \pi i \xi})$, where $A^{''} (X)$ is the ``bad'' factor associated to $A(X)$ in Definition \ref{A1234}; the function $\phi^{''}_B$ is similarly defined. To extend the proof in \cite{BLV} for sets $A$ and $B$ as in Theorem \ref{Favard-two-primes}, it suffices to prove that the following holds under the assumptions of the theorem: given $t\in[\frac{1}{2},1]$, and given a large number $m$ depending on $n$ (in \cite{BLV}, $m$ is at most logarithmic in $n$),
there exists a Borel set $\Gamma =\Gamma(t,m) \subset [0,1]$ satisfying the estimates
\begin{equation}\label{BLV-diff-gamma}
\Gamma - \Gamma \subset \bigg\{\xi : \bigg|\prod\limits_{k = 0}^{m - 1} \phi_A^{''} (L^k \xi) \cdot \phi^{''}_B (L^k t \xi)\bigg| \geq L^{-C_1 m} \bigg\},
\end{equation}
\begin{equation}\label{BLV-size-gamma}
\big| \Gamma \big| \geq C_2 L^{- (1 - \epsilon)m} ,
\end{equation}
where 
$C_1, C_2,\epsilon > 0$ are constants independent of $N$ and $m$. The number 
$t = \tan (\theta)$ depends on the direction of an appropriately selected one-dimensional projection of $S_N$.  We refer to any set $\Gamma \subset [0,1]$ satisfying (\ref{BLV-diff-gamma})--(\ref{BLV-size-gamma}) as an \textbf{SLV set} for the function $\phi^{''}_t (\xi) = \phi^{''}_A (\xi)\phi^{''}_B (t \xi)$.  Thus, to prove Theorem \ref{Favard-two-primes}, it suffices to construct an SLV set $\Gamma \subset [0,1]$ as above, provided that $A,B \subset \nn$ satisfy the number theoretic assumptions of the theorem.

For $A \subset \nn$,  define $S_A$ and $s_A$ as in (\ref{def-SA}), with $S_A=:\{s_1,\dots,s_J\}$. Let also
$$\Sigma_A : =  \{x\in [0,1]: \ \Phi_s(e^{2\pi ix})=0 \textrm{ for some } s \in S_A \}.$$
Observe that we cannot have $\Phi_1|A$. Indeed, $\Phi_1(X)=X-1$, and if this were a factor of $A(X)$, we would have $A(1)=0$, a contradiction. In particular, we have $\{0,1\}\cap\Sigma_A=\emptyset$.

We then have the following single-scale SLV result for $A$ (see Section \ref{clusterdiv} for a proof).

\begin{proposition}\label{lb-general-gamma}
Let $A \subset \nn$, and let $s_A=\prod_{i=1}^K p_i^{n_i}$, where $p_1,\dots,p_K$ are distinct primes. For each $s_j \in S_A$, let $s_j = \prod_{i = 1}^{K} p_i^{\alpha_{i,j}}$ be its prime factor decomposition. Assume that there exists an index $i\in\{1,\dots,K\}$ such that $\alpha_{i,j} \geq 1$ (so that $p_i|s_j$) for all $j\in\{1,\dots,J\}$. 
Let $E_i := | \{\alpha_{i,1}, ..., \alpha_{i, J} \} |$ (that is, $E_i$ counts the number of \textbf{distinct} exponents appearing on the prime factor $p_i$ among the $s_j \in S_A$).
Then for any $0<\lambda_i<p_i^{-1}$
 there exists a 1-periodic set $\Gamma_A \subset \mathbb{R}$ satisfying
\begin{equation}\label{lb-distance-gamma}
\dist(\Gamma_A -\Gamma_A ,\Sigma_A)>0,
\end{equation}
\begin{equation}\label{lb-size-gamma}
\big|[0, 1] \cap \Gamma_A \big| > \lambda_i^{E_i}.
\end{equation}
 
\end{proposition}

To deduce Theorem \ref{Favard-two-primes} from Proposition \ref{lb-general-gamma}, we will use the proposition to find a multiscale SLV set $\Gamma$ avoiding small values of both $\phi''_A$ and $\phi''_B$ as in (\ref{BLV-diff-gamma}). Then we must check that the size of this set is bounded from below as in (\ref{BLV-size-gamma}). This is the step where we need Theorem \ref{main theorem}. The full argument is below.

\begin{proof}[Proof of Theorem \ref{Favard-two-primes}]
Begin with two sets $A,B \subset \nn$ satisfying the assumptions of Theorem \ref{Favard-two-primes}. We are assuming that lcm$(S_A)$ has only two prime factors, say $p_i$ and $p_j$. Applying the lower bound of Theorem \ref{main theorem} to $|A|$, we get
\begin{equation}\label{size of AB}
|A|\geq p_i^{E_{i}},
\end{equation}
where $E_i$ is defined as in Proposition \ref{lb-general-gamma}. By the definition of $S_A$, $p_i$ cannot divide $|A|$, so that the inequality in (\ref{size of AB}) must be strict. We may therefore choose $\lambda_i>0$ so that
\begin{equation}\label{choice of lambda}
\frac{1}{|A|}<\lambda_i^{E_i}<p_i^{-E_i}.
\end{equation}
Consider the associated set $\Gamma_A \subset \mathbb{R}$ provided by Proposition \ref{lb-general-gamma}.  Let  
$
\nu_A:=\big|[0, 1] \cap \Gamma_A \big|
$
be the density of $\Gamma_A$ in $[0,1]$. By (\ref{lb-size-gamma}) and (\ref{choice of lambda}), we have
\begin{equation}\label{size of nu-A}
\nu_A > \lambda_i^{E_i} > \frac{1}{|A|}.
\end{equation}
Similarly, applying the same construction to $B$, we produce a set $\Gamma_B \subset \mathbb{R}$ satisfying the conclusions of Proposition \ref{lb-general-gamma} and such that $\nu_B:=\big|[0, 1] \cap \Gamma_B \big|$ satisfies 
\begin{equation}\label{size of nu-B}
\nu_B > \frac{1}{|B|}.
\end{equation}
From equation \eqref{lb-distance-gamma}, there are constants $c_A, c_B > 0$ such that 
\begin{equation}
|\phi^{''}_A (\xi) |> c_A \quad \forall \xi \in \Gamma_A - \Gamma_A
\end{equation}
and similarly for $|\phi^{''}_B|$ on $\Gamma_B - \Gamma_B$.  Setting $\Gamma_{k,A} = L^{-k} \cdot \Gamma_A$, we then have $|\phi_A^{''} (L^k \xi) |> c_A$ on $\Gamma_{k,A}$, and similarly $|\phi_B^{''} (L^{k} \xi) |> c_B$ on $\Gamma_{k,B} = L^{-k} \cdot \Gamma_B$.

Fix a large integer $R > 0$. The same pigeonholing argument as in \cite[Proposition 5.1]{BLV} (see also the proof of Lemma \ref{q-lemma3} in this article) furnishes translation parameters $\tau_{k,A}, \tau_{k,B} \in [0,R]$ so that the set
\begin{equation}\label{favard-our-gamma-entire}
\Gamma : = \bigcap_{k = 0}^{m - 1} \big( \Gamma_{k,A}- \tau_{k,A} \big) \cap \bigcap_{k = 0}^{m-1} \big(t^{-1} \Gamma_{k, B} - \tau_{k, B} \big)
\end{equation}
satisfies the inequality
\begin{equation}\label{favard-lower-bound-explicit}
\big| \Gamma \cap [0,1] \big| \geq \bigg( \frac{(R - 1)(R-t^{-1})}{R^2} \nu_A \nu_B \bigg)^m.
\end{equation}
Since the inequalities in (\ref{size of nu-A}) and (\ref{size of nu-B}) are strict,  we may choose $R > 0$ to be large enough so that
$$
\frac{(R - 1)(R-t^{-1})}{R^2} \nu_A \nu_B  \geq L^{-(1 - \epsilon)}
$$
for some $\epsilon > 0$.
Note that our choice of $R = R(m,t)$ may depend upon $m$ (therefore $n$) and $t$; however, this does not affect the rest of the argument. 

Thus, the set $\Gamma$ defined in (\ref{favard-our-gamma-entire}) 
satisfies \eqref{BLV-size-gamma}. Moreover, since
$$
\Gamma_{k,A} - \Gamma_{k,A} = \big( \Gamma_{k,A} - \tau_{k,A} \big) - \big(\Gamma_{k,A} - \tau_{k,A} \big),
$$
and similarly for $\Gamma_{k, B} - \Gamma_{k,B}$, for each $0 \leq k \leq m - 1$ we have
$$
|\phi^{''}_A (L^{k} \xi)|,  \, |\phi^{''}_B (t L^k \xi)| > c,  \quad \forall \xi \in \Gamma - \Gamma,
$$
where $c = \min \{c_A, c_B \}$.  Hence,  \eqref{BLV-diff-gamma} holds with $C_1 = \frac{2 \log ( 1 /c )}{\log L}$, and our choice of $C_1 > 0$ is independent of $n$. Thus, the set $[0,1] \cap \Gamma$ with $\Gamma$ given by \eqref{favard-our-gamma-entire} is an SLV-set for the function $\phi^{''}_A(x) \phi^{''}_B(tx)$, as required.
\end{proof}


\section{Constructing Single-Scale Sets $\Gamma$}\label{SLV-section}

The main new idea in the proof of Proposition \ref{lb-general-gamma} is the following ``cluster reduction''. 
Let $A \subset \nn$ be a finite set, and define $S_A$ and $s_A$ as in (\ref{def-SA}). Then we may split $S_A$ into subsets called \textbf{clusters}, construct an SLV set associated to each cluster, and then take the intersections of appropriate translates of them to get the set $\Gamma_A$ in the proposition. 

The results of Section \ref{oneclustergamma} and \ref{manyclusters} apply to any finite set $A\subset\NN$ and any splitting of $A$ into clusters.
In Section \ref{oneclustergamma}, we follow a number-theoretic approach due to Matthew Bond and the first author (cf. \cite[Conjecture 4.6]{L1}), which extends slightly that of \cite{BLV}. We should note here that Conjecture 4.6 in \cite{L1} turns out to be false, with a counterexample provided here in Section \ref{example1}. Therefore, if we simply tried to use the construction in Section \ref{oneclustergamma} with $A=\mathcal{C}$ as a single cluster, our quantitative bounds on $\Gamma_A$ would not be good enough. However, we can use the same construction more efficiently if we split up $A$ into smaller clusters, construct a cluster-dependent set $\Gamma(\calC)$ for each one, then take the intersection of appropriate translates of the sets thus obtained.
 
In Section \ref{clusterdiv}, we set up the cluster splitting that provides the requisite quantitative estimate \eqref{lb-size-gamma} in Proposition \ref{lb-general-gamma}.
This part requires the additional assumption (stated in the proposition) on the prime factorization of the elements of $S_A$.

\subsection{A single-cluster SLV set}\label{oneclustergamma}
Let $s_1,\dots,s_J$ be an enumeration of the distinct elements of $S_A$. 

\begin{definition}
A subset $\mathcal{C} \subset S_A$ is called a \textbf{cyclotomic divisor cluster}, or (for short) a \textbf{cluster}, of $A$.
\end{definition}

Fix a cluster $\mathcal{C}\subset S_A$. Relabelling the elements of $S_A$ if necessary, we may assume that $\mathcal{C} = \{s_1, ..., s_I \}$ for some $I\leq J$. Let $N = lcm(s_1, \cdots, s_I)$. Suppose that we can write $N =QU$,
where 
\begin{equation}\label{q-e1}
s_j \hbox{ does not divide } Q \hbox{ for any } j\in \{1,\dots,I\}.
\end{equation}
In particular, it follows that $(s_j ,U)>1$ for each $s_j \in \mathcal{C}$. It will be to our advantage to choose $Q$ as large as possible.

Each $s_j \in \mathcal{C}$ can be written as $s_j = r_jt_j$, where $r_j:=(s_j,Q)$ and $t_j:= s_j/(s_j,Q)$.  Let $T := \max(t_1,\dots,t_I)$.
Define
\begin{equation}\label{general-gamma}
\Gamma (\mathcal{C},\rho) =\left\{\xi\in \RR :\ \dist(\xi,\frac{1}{Q}\zz)<\frac{\rho}{2}\right\},                                                                                            \end{equation}
where $0<\rho<(QT)^{-1}$. The next two lemmas guarantee that the set in (\ref{general-gamma}) has the properties we need. Specifically, the required arithmetic structure of $\Gamma (\mathcal{C},\rho)$ is verified in 
Lemma \ref{q-lemma1}, and Lemma \ref{q-lemma2} furnishes a lower bound on $|[0,1] \cap \Gamma (\mathcal{C},\rho)|$.

\begin{lemma}\label{q-lemma1}
The set $\Gamma := \Gamma (\mathcal{C},\rho)$ defined above satisfies
$$
\text{dist} \big( \Gamma -  \Gamma, \Sigma(\mathcal{C}) \big) > 0
$$
where
\begin{equation}\label{q-cluster-sigma}
\Sigma(\mathcal{C}) : =  \{\xi \in \RR: \ \Phi_{s_j} (e^{2\pi i\xi})=0 \textrm{ for some } s_j\in \mathcal{C} \}.
\end{equation}
\end{lemma}
\begin{proof}
It suffices to prove that $\Gamma-\Gamma$ is disjoint from $\Sigma(\calC)$. The conclusion then follows by starting with a slightly larger $\rho$ that still satisfies $0<\rho<(QT)^{-1}$, and then shrinking it a little bit.

Let $\xi\in\Sigma(\calC)$. Then there is an $s_j\in \mathcal{C}$ such that 
$\Phi_{s_j}(e^{2\pi i\xi})=0$, so that $\xi=b/s_j$ for some $b\in\zz$ with $(b,s_j)=1$. 
Suppose $\xi\in\Gamma-\Gamma$, then there is an integer $a$ such that
$$
\left|\frac{a}{Q}-\frac{b}{s_j}\right|
=\left|\frac{a}{Q}-\frac{b}{t_jr_j}\right|<\rho<\frac{1}{QT}.
$$
Multiply this by $Qt_j$:
$$
\left| at_j-\frac{bQ}{r_j}\right|<\frac{t_j}{T}\leq 1.
$$
But $r_j|Q$, so that $bQ/r_j$ is integer. Therefore 
$$at_j=\frac{bQ}{r_j},$$
and in particular $t_j$ divides $bQ/r_j$. Since $(b,s_j)=1$, we also have $(b,t_j)=1$, so that $t_j$ divides $Q/r_j$.
But then $s_j=t_jr_j$ divides $Q$, contradicting (\ref{q-e1}).
\end{proof}

\begin{lemma}\label{q-lemma2}
Let $\Gamma=\Gamma (\mathcal{C},\rho)$ be as above, and let $0<\lambda<T^{-1}$. Then $\Gamma$ is 1-periodic, and there exists a choice of $0<\rho<(QT)^{-1}$ such that
\begin{equation}\label{e-lambda}
\big| [0,1] \cap \Gamma  \big| > \lambda.
\end{equation}
\end{lemma}

\begin{proof}
The periodicity is clear from the definition. Let $\rho\nearrow (QT)^{-1}$, then
$$
|[0,1] \cap \Gamma | = Q \rho\nearrow \frac{Q}{Q T}=\frac{1}{T}.
$$
Thus it suffices to choose $\rho$ sufficiently close to $(QT)^{-1}$.
 \end{proof}

 \subsection{Combining several clusters}\label{manyclusters}

\begin{lemma}\label{q-lemma3}
Suppose that $\mathcal{C}^1, ..., \mathcal{C}^k$ are clusters associated to some $A \subset \nn$. For each $l\in\{1,\dots,k\}$, let $\Gamma^l:=\Gamma(\calC^l,\rho^l)$ be the set defined in \eqref{general-gamma} and satisfying (\ref{e-lambda}), with the corresponding parameters $Q_l, T_l, \rho_l, \lambda_l$. Then there exist translation parameters $\tau_1, ..., \tau_k \in [0,1]$ such that
\begin{equation}\label{q-e4}
\bigg| [0,1] \cap  \bigcap_{l=1}^{k} \big( \Gamma^l+ \tau_l \big)  \bigg| > \prod_{l = 1}^{k} \lambda_l,
\end{equation}
Furthermore, if we define $\Gamma^{1,\dots,k}:=  \bigcap_{l=1}^{k} \big( \Gamma^l+ \tau_l \big)$ with this choice of $\tau_l$, then
\begin{equation}\label{q-e4a}
\text{dist} \big( \Gamma^{1,\dots,k} -  \Gamma^{1,\dots,k}, \bigcup_{l=1}^k \Sigma(\mathcal{C}^l) \big) > 0
\end{equation}
where $\Sigma(\mathcal{C}^l)$ is defined as in (\ref{q-cluster-sigma}) with $
\mathcal{C}=\mathcal{C}^l$.

\end{lemma}

\begin{proof}
We first note that for any fixed $l$, and for any choice of $\tau_1,\dots,\tau_k$,
$$
\text{dist} \big( \Gamma^{1,\dots,k} -  \Gamma^{1,\dots,k},  \Sigma(\mathcal{C}^l) \big)
\geq \text{dist} \big( \Gamma^l -  \Gamma^l, \Sigma(\mathcal{C}^l) \big)
>0
$$
by Lemma \ref{q-lemma1}. Hence (\ref{q-e4a}) holds for any choice of the parameters $\tau_l$.

We now prove (\ref{q-e4}).
The proof uses essentially the same argument as the proof of \cite[Proposition 5.1]{BLV}, except that the large parameter $M$ is not needed since all the sets $\Gamma^l$ are 1-periodic.  For each $x \in [0, 1)$, consider the function
$$
\Psi (x) : = \int_{[0,1]^k}  \prod_{l = 1}^{k} \mathbf{1}_{\Gamma^l} (x - \tau_l)   d \tau_1 \cdots d \tau_k,
$$
where $\tau_l \in [0,1]$ are independent translation parameters, and the addition is mod 1.
Using this independence, we evaluate $\Psi (x)$ as a product of single-variable averages. We have
$$
\int_0^1 \mathbf{1}_{\Gamma^l} (x - \tau_l) d \tau_l = \big| [0, 1] \cap \Gamma^l \big| > \lambda_l,
$$
by (\ref{e-lambda}) and the 1-periodicity of each $\Gamma^l$. This leads to the pointwise lower bound
$$
\Psi (x) >  \prod_{l} \lambda_l .
$$
Integrating and applying Fubini's theorem, we get that
$$
\prod_l \lambda_l < \int_0^1 \Psi (x) dx 
= \int_{[0,1]^k} \bigg|[0,1] \cap  \bigcap_l \big( \Gamma^l+ \tau_l \big)  \bigg| d \tau_1 \cdots d \tau_k,
$$
In particular, there exist translation parameters $\tau_1, \dots, \tau_k \in [0,1]$ so that
$$
\bigg|[0,1] \cap  \bigcap_l \big( \Gamma^l +\tau_l \big) \bigg| > \prod_l \lambda_l,
$$
as claimed.
\end{proof}

\subsection{Choosing the clusters}\label{clusterdiv}

We divide the set $S_A$ into clusters based on their number-theoretic properties.  Lemma \ref{q-lemma3} then produces an SLV set associated to these clusters, with an appropriate lower bound on its measure. We first introduce notation that allows us to partition the set $S_A : = \{s_1, ..., s_J \}$ in a useful way.

Each $s_j \in S_A$ has the form $s_j = p_1^{\alpha_{1,j}} \cdot p_2^{\alpha_{2,j}} \cdots p_K^{\alpha_{K,j}}$ where $p_1, ..., p_K$ are the distinct prime divisors of $s_A$. For each $1 \leq i \leq K$, let
$$
\textsf{EXP} (i) : = \{\alpha_{i,1},\dots, \alpha_{i,J}\}.
$$
In words, $\textsf{EXP}(i)$ contains the exponents appearing on $p_i$ among the $s_j \in S_A$. Observe that some of the numbers $\alpha_{i,1},\dots,\alpha_{i,J}$ may be equal.
We set $\# \textsf{EXP} (i) = E_i$, the number of \textbf{distinct} exponents in  
$\textsf{EXP} (i) $.

\begin{definition}\label{power-cluster-defn}
Fix  $1 \leq i \leq K$. Then, for each $\alpha \in \textsf{EXP}(i)$,  define
$$
\mathcal{C}^{i, \alpha} = \{s_j \in S_A : s_j = p_i^{\alpha} \cdot q, \textrm{ for some } q \in \mathbb{N} \textrm{ with } (p_i, q) = 1\}
$$
That is, $\mathcal{C}^{i,\alpha}$ is the cluster of $s_j$ such that $\alpha_{i,j}=\alpha$.
\end{definition}

We have the following lemma concerning clusters associated to non-zero $\alpha \in \mathsf{EXP}(i)$. 

\begin{lemma}\label{power-cluster-lemma}
Fix $1 \leq i \leq K$. Let $\alpha \in \mathsf{EXP}(i)\setminus\{0\}$, and let $\calC:= \mathcal{C}^{i,\alpha}$. Let $0<\lambda_i<p_i^{-1}$.
Then there is a choice of $Q = Q^{i, \alpha}$ and $\rho=\rho^{i,\alpha}$
such that the set
$\Gamma_{i,\alpha}:= \Gamma(\mathcal{C}^{i,\alpha},\rho^{i,\alpha})$ constructed in Section \ref{oneclustergamma} satisfies the estimate 
$$|[0,1] \cap \Gamma_{i,\alpha} | > \lambda_i.
$$
\end{lemma}

\begin{proof}
Since $i$ and $\alpha$ are fixed throughout the proof, we suppress them for now and use the notation of Section \ref{oneclustergamma} with $\calC= \mathcal{C}^{i,\alpha}$.

Relabelling the elements of $S_A$ if necessary, we may assume that $\calC=\{s_1,\dots,s_I\}$ for some $I\leq J$. 
Each $s_j\in\calC$ can be written as 
$s_j = p_i^{\alpha} q_j$, where $ (q_j,p_i)=1.$ We then define
$$
Q:=  p_i^{\alpha-1} \hbox{\, lcm}(q_1,\dots,q_I).
$$
Then for each $j\in\{1,\dots,I\}$ we have $(s_j,Q)= p_i^{\alpha-1} q_j$ and
$$
t_j=\frac{s_j}{(s_j,Q)}=p_i.
$$
It follows that we can take 
$T= \max(t_j)=p_i$. 

Let $\Gamma_{i,\alpha}:=\Gamma (\mathcal{C},\rho)$ be the set defined in \eqref{general-gamma}. Then 
we may choose $0<\rho<(QT)^{-1}$ so that
the conclusions of Lemmas \ref{q-lemma1} and \ref{q-lemma2} hold with $\lambda=\lambda_i$, as claimed.
\end{proof}

\begin{proof}[Proof of Proposition \ref{lb-general-gamma}] 
We now assume that $A\subset\NN$ satisfies the hypotheses of the proposition.
In particular, we have $0\not\in \mathsf{EXP}(i)$; this is simply a rephrasing of the requirement that $\alpha_{i,j} \geq 1$. 

Let $\lambda_i<p_i^{-1}$. For each $\alpha \in \mathsf{EXP}(i)$, let $\Gamma_{i,\alpha}$ be the set provided by Lemma \ref{power-cluster-lemma}.   Applying Lemma \ref{q-lemma3}, we find translation parameters $z_{\alpha}$ so that
\begin{equation}\label{q-cluster-estimate-gamma}
\big| [0,1] \bigcap_{\alpha \in \mathsf{EXP(i)}}(\Gamma _{i,\alpha} + z_{\alpha}  \big)\big| > \lambda_i^{E_i},
\end{equation}
Let $\Gamma_A : = \bigcap_{\alpha}(\Gamma_{i,\alpha} + z_{\alpha} \big)$; then \eqref{q-cluster-estimate-gamma} shows that $\Gamma_A$ has the correct size.  
Moreover, since $\Sigma_A=\bigcup_{\alpha\in \mathsf{EXP}(i)}
 \Sigma(\mathcal{C}^{i,\alpha})$, 
by (\ref{q-e4a}) we have
$$
\text{dist} \big( \Gamma_A -  \Gamma_A, \Sigma_A \big) > 0,
$$
where $\Sigma(\mathcal{C}^{i,\alpha})$ is defined as in (\ref{q-cluster-sigma}) with $
\mathcal{C}=\mathcal{C}^{i,\alpha}$.
This proves the proposition.
\end{proof}

\begin{remark}
The cluster splitting above is sufficient for our purposes if $K=2$. In this case, since each element of $S_A$ must have at least two distinct prime factors (cf. Remark \ref{remark1}), we must have $\alpha_{i,j}\geq 1$ for each $i\in\{1,2\}$ and each $j$. Thus the construction above works with both choices of $i\in\{1,2\}$. If $K\geq 3$, we would not be able to assume that, but we could still construct an SLV set by splitting the elements of $S_A$ into disjoint sets $S_{A,i}$ such that $p_i|s$ for each $s\in S_{A,i}$, applying the construction in Section \ref{clusterdiv} to each such subset, and then proceeding as in Lemma \ref{q-lemma3} to take the intersection of appropriate translates of the sets thus obtained. We expect that optimizing over constructions of this type should generate SLV sets that approach the maximal allowed size. However, in the general case, we do not know how to prove matching lower bounds on the size of $A$. See Section \ref{cuboid-lower-bound} for further discussion.
\end{remark}


\section{Cyclotomic divisibility tools}\label{cyclo-tools}

In this section, we develop the tools needed to prove our results on cyclotomic factor decompositions and vanishing sums of roots of unity. Some of the notation here has been borrowed from \cite{LaLo1} and adapted to our setting.

\subsection{Multisets}
 
We will work in the ambient group $\ZZ_M$, where $M=p_1^{n_1}\dots p_K^{n_K}$, $p_1,\dots,p_K$ are distinct primes, and $n_1,\dots,n_K\in\NN$. We will use $m$ and $N$ (possibly with subscripts) to denote divisors of $M$. For the purpose of proving Theorem \ref{main theorem}, it would be sufficient to consider the case $K=2$. The discussion in Section \ref{sec-examples} will require the more general notation. 

We use $A(X)$, $B(X)$, etc. to denote polynomials modulo $X^M-1$ with integer coefficients. 
Each such polynomial $A(X)=\sum_{a\in\ZZ_M} w_A(a) X^a$ is associated with a weighted multiset in $\ZZ_M$, which we will also denote by $A$, with weights $w_A(x)$ assigned to each $x\in\ZZ_M$. (If the coefficient of $X^x$ in $A(X)$ is 0, we set $w_A(x)=0$.) In particular, if $A$ has $\{0,1\}$ coefficients, then
$w_A$ is the characteristic function of a set $A\subset \ZZ_M$. We will use $\calm(\ZZ_M)$ to denote the 
family of all weighted multisets in $\ZZ_M$, and reserve the notation $A\subset \ZZ_M$ for sets.
We will also use $\calm^+$ to denote the family of all weighted multisets in $\ZZ_M$ with nonnegative weights:
$$
\calm^+=\{A\in\calm(\ZZ_M): \ w_A(a)\geq 0\hbox{ for all }a\in\ZZ_M\}.
$$
Let $A\in\calm^+(\ZZ_M)$, with the corresponding mask polynomial $A(X)$. We use $|A|$ to denote the cardinality of $A$ with multiplicity, so that $|A|=A(1)=\sum_{x\in\ZZ_M}w_A(x)$.
If $\Lambda\subset\ZZ_M$, we use $A\cap\Lambda$ to denote intersection with multiplicity, so that $w_{A\cap\Lambda}(x)=w_A(x)w_\Lambda(x)$.

If $N|M$, then any $A\in \calm(\ZZ_M)$ induces a weighted multiset $A$ mod $N$ in $\ZZ_N$, with the corresponding mask polynomial $A(X)$ mod $(X^N-1)$ and induced weights 
\begin{equation}\label{induced-weights}
w_A^N(x)= \sum_{x'\in\ZZ_M: x'\equiv x\,{\rm mod}\, N} w_A(x'),\ \ x\in\ZZ_N.
\end{equation}
We extend the multiset notation to $\ZZ_N$, so that for example $\calm(\ZZ_N)$ and $\calm^+(\ZZ_N)$ denote the appropriate families of multisets.

We use convolution notation $A*B$ to denote the weighted sumset of $A,B\in\calm(\ZZ_M)$:
 $$
 (A*B)(X)=A(X)B(X),\ \ 
 w_{A*B}(x)=(w_A*w_B)(x)=\sum_{y\in\ZZ_M} w_A(x-y)w_B(y).$$
 If one of the sets is a singleton, say $A=\{x\}$, we write $x*B=\{x\}*B$.  

\subsection{Grids and fibers}\label{grids-fibers}
We encourage the reader to use the geometrical interpretation from \cite{LaLo1}, based on the Chinese Remainder Theorem. Specifically, let $M_i=M/p_i^{n_i}$ for $i=1,\dots,K$. Then any $x\in\ZZ_M$ can be written uniquely as 
$$
x\equiv \sum_{i=1}^{K} x_iM_i\hbox{ mod }M,\hbox{ where }x_i\in\{0,1,\dots,p_i-1\}.
$$
Thus $\ZZ_M$ can be thought of as a $K$-dimensional lattice $\ZZ_{p_1^{n_1}}\oplus\dots\oplus \ZZ_{p_K^{n_K}}$. In this interpretation, $(x_1,\dots,x_K)$ are the coordinates of $x\in\ZZ_M$, and each of the primes $p_1,\dots,p_K$ corresponds to one of the cardinal directions. A similar picture, possibly with fewer directions, applies to $\ZZ_N$ with $N|M$. 

For $D|N|M$, a {\em $D$-grid} in $\ZZ_N$ is a set of the form
$$
\Lambda^N(x,D):= x*D\ZZ_N=\{x'\in\ZZ_N:\ D|(x-x')\}
$$
for some $x\in\ZZ_N$.  
If $N=\prod_{j=1}^K p_j^{\alpha_j}$ is a divisor of $M$, with  $0\leq \alpha_j\leq n_j$, we let 
$$
D(N):= \prod_{j=1}^K p_j^{\gamma_j},\ \hbox{ where }
\gamma_j=\max(0,\alpha_j-1)\hbox{ for }j=1,\dots,K.
$$

Let $p_i$ be a prime factor of $N$. An {\em $N$-fiber in the $p_i$ direction} is a set of the form $x*F^N_i\subset\ZZ_N$, where $x\in\ZZ_N$ and
\begin{equation}\label{def-Fi}
F^N_i=\{0,N/p_i,2N/p_i,\dots,(p_i-1)N/p_i\}.
\end{equation}
Thus $x*F^N_i=\Lambda^N(x,N/p_i)$.


\subsection{Cuboids and structure results}
As before, we work in $\ZZ_M$, where $M=\prod_{i=1}^K p_i^{n_i}$, and let $A\in\calm(\ZZ_M)$.
We will use the following notation from \cite{LaLo1}. For multisets $\Delta\in\calm(\ZZ_N)$, where $N|M$, we define the \textbf{$\Delta$-evaluations of $A$ in $\ZZ_N$:}
\begin{equation}\label{delta-eval}
\bbA^N[\Delta]=\sum_{x\in\ZZ_N}w_A^N(x)w_\Delta^N(x).
\end{equation}
The following special case is of particular interest.

\begin{definition}
\label{def-N-cuboids}
Let $M$ and $N$ be as above, so that $N=\prod_{i=1}^K p_i^{n_i-\alpha_i}$, with $0\leq \alpha_i\leq n_i$ for each $i=1,\dots,K$. An \textbf{$N$-cuboid} is a multiset $\Delta \in \calm (\zz_N)$ associated to a mask polynomial of the form
\begin{equation}\label{def-delta}
\Delta(X)= X^c\prod_{j:p_j|N} (1-X^{d_jN/p_j})
\end{equation}
 with $(d_j,p_j)=1$ for all $j$. 
\end{definition}

The geometric interpretation of $N$-cuboids $\Delta$ is as follows. With notation as in Definition \ref{def-N-cuboids}, 
recall that $D(N)=N/\prod_{j:p_j|N}p_j$. Then the ``vertices'' $x\in\ZZ_N$ with $ w^N_\Delta(x)\neq 0$ form a full-dimensional rectangular box in the grid $\Lambda(c,D(N))$, with one vertex at $c$ and alternating $\pm 1$ weights.

The following cyclotomic divisibility test has been known and used previously in the literature. The equivalence between (i) and (iii) is the Bruijn-R\'edei-Schoenberg theorem on the structure of vanishing sums of roots of unity (see \cite{deB}, \cite{LL}, \cite{Mann}, \cite{Re1}, \cite{Re2}, \cite{schoen}). For the equivalence (i) $\Leftrightarrow$ (ii), see e.g.  \cite[Section 3]{Steinberger}, \cite[Section 3]{KMSV}.

\begin{proposition}\label{cuboid}
Let $A\in\calm(\ZZ_N)$. Then the following are equivalent:

\medskip

(i) $\Phi_N(X)|A(X)$,

\medskip

(ii) For all $N$-cuboids $\Delta$, we have
\begin{equation}\label{id-3a}
\bbA^N[\Delta]=0,
\end{equation}

\medskip

(iii) $A$ mod $N$ is a linear combination of $N$-fibers, so that
$$A(X)=\sum_{i:p_i|N} P_i(X)F^N_i(X) \mod X^N-1,$$
where $P_i(X)$ have integer (but not necessarily nonnegative) coefficients.
\end{proposition}

Proposition \ref{cuboid} can be strengthened as follows if $N$ has only two distinct prime factors. This goes back to the work of de Bruijn \cite{deB}; a self-contained proof is provided in \cite[Theorem 3.3]{LL}.

\begin{lemma}\label{structure-thm} 
Let $A\in\calm^+(\ZZ_N)$.
Assume that $\Phi_{N}|A$, where $N$ has two distinct prime factors $p_1,p_2$. Then $A$ mod $N$ is a linear combination of $N$-fibers with nonnegative weights. In other words,
$$A(X)=P_1(X)F^N_1(X)+ P_2(X)F^N_2(X) \mod X^N-1,$$
where $P_1,P_2$ are polynomials with nonnegative coefficients.

\end{lemma}

\begin{lemma}\label{grid-split}
Assume that $N|M$ and $A\in\calm(\ZZ_N)$. Let $m|D(N)$.Then $\Phi_N|A$ if and only if $\Phi_N|(A\cap\Lambda)$ for every $m$-grid $\Lambda$.
\end{lemma}

\begin{proof}
If $\Delta$ is an $N$-cuboid with one vertex $z\in\ZZ_N$, then all its vertices are contained in $\Lambda(x,D(N))$. Since $m|D(N)$, it follows that any $m$-grid $\Lambda$ containing any vertex of $\Delta$ must contain all of its vertices. The lemma now 
follows from the equivalence (i) $\Leftrightarrow$ (ii) in Proposition \ref{cuboid}.
\end{proof}


\section{Proof of Theorem \ref{main theorem}}\label{main-proof}

We are now ready to prove Theorem \ref{main theorem}. For the reader's convenience, we state it here again in the notation of Section \ref{cyclo-tools}.

\begin{proposition}\label{2primebound}
Assume that $K=2$, and write $p=p_1,q=p_2$ for short.
Let $A\in\calm^+(\ZZ_M)$, and let $m_1,m_2,\dots,m_r$ be divisors of $M$ such that $m_j=p^{\alpha_j} q^{\beta_j}$, where
$$
1\leq\alpha_1<\alpha_2<\dots <\alpha_r.
$$
Assume that $\Phi_{m_1}\Phi_{m_2}\dots \Phi_{m_r}|A$, and that $q\nmid |A|$.
Then $|A|\geq p^r$.
\end{proposition}

\begin{proof}
We proceed by induction in $r$. For the base case, suppose that $r=1$. By Lemma \ref{structure-thm}, $A$ mod $m_1$ is a union of $m_1$-fibers in the $p$ and $q$ directions. Since $q\nmid |A|$, at least one of these fibers must be in the $p$ direction. Hence $|A|\geq p$. 

Suppose now that $r\geq 2$, and that the proposition is true with $r$ replaced by $r-1$. 
Let $D:=D(p^{\alpha_1})=p^{\alpha_1-1}$. 
We write $\ZZ_M$ as a disjoint union of grids $\Lambda(y_i,D)$, where $y_0:=0,y_1,\dots,y_{D-1}\in\ZZ_M$. 
Let $A_i=A\cap \Lambda(y_i,D)$.
By Lemma \ref{grid-split}, we have $\Phi_{m_1}\Phi_{m_2}\dots \Phi_{m_r}|A_i$ for each $i$. Moreover, since $q\nmid |A|$, there exists at least one $i$ such that $q\nmid |A_i|$. Without loss of generality, we may assume that $q\nmid|A_0|$, We will prove that $|A_0|\geq p^r$.

Write $m=m_1$ for short.
Applying Lemma \ref{structure-thm} to $A_0$ on the scale $m$, we see that $A_0$ mod $m$ is a linear combination of $m$-fibers in the $p$ and $q$ directions with nonnegative coefficients. Taking into account that $A_0\subset D\ZZ_M$, we see that
\begin{equation}\label{A0-decomp}
A_0(X)\equiv P_1(X^D)F^{m}_1(X)+ P_2(X^D)F^{m}_2(X) \mod X^{m}-1,
\end{equation}
where $P_1,P_2$ are polynomials with nonnegative coefficients, $F^{m}_1(X)=1+X^{m/p}+\dots +X^{(p-1)m/p}$ is an $m$-fiber in the $p=p_1$ direction,
and $F^{m}_2=1+X^{m/q}+\dots+X^{(q-1)m/q}$ is an $m$-fiber in the $q=p_2$ direction. 

Moreover, we have the following simplification. For each $a\in\ZZ_M$, we may use the Chinese Remainder Theorem to write 
\begin{equation}\label{decomp-a}
aD\equiv a_1p^{\alpha_1} + a_2m/p\mod m,
\end{equation}
 where $a_1,a_2\in\ZZ_M$. Then
\begin{align*}
X^{aD}F^m_1(X)&=X^{a_1p^{\alpha_1}}X^{a_2m/p}
\Big(1+X^{m/p}+\dots +X^{(p-1)m/p}\Big)
\\
&= X^{a_1p^{\alpha_1}}\Big(X^{a_2m/p}+X^{(a_2+1)m/p}+\dots +X^{(a_2+p-1)m/p}\Big)
\\
&\equiv X^{a_1p^{\alpha_1}}F^m_1(X) \mod X^m-1.
\end{align*}
Applying this to every monomial in $P_1(X^D)$ if necessary, we may assume that $P_1$ satisfies
$$
P_1(X^D)=P_0(X^{p^{\alpha_1}}),
$$
where $P_0(X)$ is a polynomial with nonnegative coefficients.
Since $q\nmid |A_0|$ and $|F^{m}_2|=q$, it follows that 
\begin{equation}\label{qndiv P1}
q\nmid P_1(1).
\end{equation}

We now split up $A_0$ further, as follows. 
Let $x_0:=0,x_1,\dots,x_{p-1}$ be points in $\Lambda(0, D)$ such that $(x_i-x_j,M)=M/p^{\alpha_1-1}$ for $i\neq j$.
Using the Chinese Remainder Theorem as in (\ref{decomp-a}), we write 
$\Lambda(0, D)=\bigcup_{j=0}^{p-1}\Lambda_j$, where $\Lambda_j=\Lambda(jm/p,pD)$. Let
$A_{0,j}=A_0\cap\Lambda_j$.

The geometric idea in the next step is as follows. Think of $p$ and $q$ as two directions in a plane. We may then interpret $\{\Lambda_j\}$ as a decomposition of a 2-dimensional grid $\Lambda(0,D)$ into a system of parallel lines, each perpendicular to the $p$ direction. 
Consider the decomposition (\ref{A0-decomp}). The fibers in the $p$ direction are ``orthogonal'' to the parallel lines, so each such fiber contributes one point to each line. The fibers in the $q$ direction are parallel to $\Lambda_j$, hence any translated copy of $F^{m}_2$ is either contained fully in $\Lambda_j$ or disjoint from it.

We now write this out more explicitly. For the first part of (\ref{A0-decomp}), we have
$$
P_1(X^D)F^{m}_1(X)= P_0(X^{p^{\alpha_1}}) \Big(1+X^{m/p}+\dots +X^{(p-1)m/p}\Big),
$$
and for any monomial $cX^{ap^{\alpha_1}}$ appearing in $P_0$, we have 
$ap^{\alpha_1}+ jm/p \in \Lambda_j$ for $j=0,1,\dots,p-1$. For the second part, we consider the multiset associated to $P_2(X^D)$, decompose it into disjoint multisets in $\calm(\Lambda_j)$ for $j\in\{0,1,\dots,p-1\}$, and use that $F^m_2\subset \Lambda_0$. Hence for each $0 \leq j \leq p - 1$, there are polynomials $Q_j$ such that
$$
A_{0,j}(X)=P_1(X) X^{x_j} + Q_j (X) F^{m}_2(X) \mod X^{m}-1.
$$
This decomposition is analogous to that of $A_0$ in \eqref{A0-decomp}.
Observe that
$$
|A_{0,j}|=A_{0,j}(1)=P_1(1) + Q_j(1)F^m_2(1)= P_1(1) + Q_j(1)q.
$$
By (\ref{qndiv P1}), $|A_{0,j}|$ is not divisible by $q$.
On the other hand, by another application of Lemma \ref{grid-split}, we have $\Phi_{m_2}\dots \Phi_{m_r}|A_{0,j}$ for each $j$. 

Applying the inductive assumption to $A_{0,j}$ for each $j$, we see that $|A_{0,j}|\geq p^{r-1}$. Hence
$$
|A_0|=\sum_{j=0}^{p-1}|A_{0,j}|\geq p^r,
$$
as claimed.
\end{proof}


\section{Examples and discussion}\label{sec-examples}

We discuss briefly a few motivating examples and the possibility of extending the results here to more general product sets. A minor inconvenience is that, in the Favard length setting, the set $S_A$ defined in (\ref{def-SA}) depends on $L$, hence on both of the sets $A$ and $B$ in (\ref{def-S}), and not just on $A$. Therefore, for the purpose of this discussion, consider the set
$$
S'_A:= \{s:\Phi_s|A\text{ and }(s,|A|)=1\}.
$$
Thus $S_A\subseteq S'_A$, with equality if $|A|$ and $|B|$ have the same prime factors. In the Favard length examples throughout the rest of this section, we will assume that $A=B$, so that $S_A=S'_A$.

\subsection{Cluster splitting is necessary.}\label{example1}
It is not always advantageous to group divisors in large clusters.  Consider the following example. Let $N=p^{10}q^{10}$, where $p,q$ are distinct primes. Let
\begin{align*}
A(X)&=\frac{(1-X^{N})(  1-X^{N/p^9} )}{(1-X^{N/p})(1-X^{N/p^{10}})}
+\frac{(1-X^{N})(  1-X^{N/q^9}  )}{(1-X^{N/q})(1-X^{N/q^{10}})}
\\
&= \Big(1+X^{N/p}+\dots +X^{(p-1)N/p}\Big)  \Big(1+X^{N/p^{10}}+\dots +X^{(p-1)N/p^{10}}\Big) 
\\
&+ \Big(1+X^{N/q}+\dots +X^{(q-1)N/q}\Big)  \Big(1+X^{N/q^{10}}+\dots +X^{(q-1)N/q^{10}}\Big) .
\end{align*}
Then $\Phi_{s_1}\Phi_{s_2}|A(X)$ for $s_1=N=p^{10}q^{10}$ and $s_2=pq$, whereas $|A|=p^2+q^2$.

Suppose we try to apply the construction of Section \ref{oneclustergamma} (Lemmas \ref{q-lemma1} and \ref{q-lemma2}) to the cluster $\mathcal{C} : = \{s_1, s_2 \}$.  To do this, we need to choose an appropriate $Q|N$ satisfying (\ref{q-e1}). Since $s_2=pq$ cannot divide $Q$, we must choose $Q$ to be a power of only one of the primes. Let us be as generous as we can in the circumstances, and choose $Q =q^{10}$, where $q$ is the larger prime. But then $(s_1,Q)=q^{10}$, so that $t_1 =p^{10}$. In order for that to be less than $|A|$, we would need $p^{10}<p^2+q^2$, which is false if $p$ and $q$ are of about the same size. For the same reason, the set $A$ does not satisfy the assumptions of Proposition \ref{single prime}. It also provides a counterexample to Conjecture 4.6 in \cite{L1}.

However,  the construction in Section \ref{manyclusters} is more efficient.  Consider the following cluster division. Let $\mathcal{C}^1 : = \{p^{10} q^{10} \}$ and $\mathcal{C}^2 = \{ pq \}$. We then choose $Q_1 = p^{9} q^{10}$ and $Q_2 = q$, with $T_1 = T_2 = p^{-1}$.  Let $\Gamma^1$ and $\Gamma^2$ be the corresponding single-cluster SLV sets constructed in Section \ref{oneclustergamma}. By Lemma \ref{q-lemma3}, there exist translation parameters $z_1, z_2 \in [0,1]$ so that:
$$
\bigg| [0,1] \cap \big( \Gamma^1+ z_1 \big) \cap \big( \Gamma^2  + z_2 \big) \bigg| \geq \frac{1}{p^2} > \frac{1}{p^2 + q^2} = \frac{1}{|A|}.
$$
This shows that dividing $S_A$ into optimal clusters is a key component of obtaining sets $\Gamma_A$ whose size compares favourably with $|A|$. 

\subsection{Explicit examples with two prime factors}
The set $A$ in Section \ref{example1} provides an explicit example of a set satisfying the two-prime assumption of Theorems \ref{main theorem} and \ref{Favard-two-primes}. Other examples can be constructed in a similar way. For instance, let $1\leq\alpha_1<\alpha_2<\dots <\alpha_k$, and let $N=p^\alpha q^\beta$, where $p,q$ are distinct primes and $\alpha\geq\alpha_k$. 
Then the ``long fiber" $\mathcal{F}$ with the mask polynomial
$$
\mathcal{F}(X)=\prod_{j=1}^k \Phi_p(X^{p^{\alpha_k-1}q^\beta})
$$
is divisible by all $\Phi_s(X)$ such that the exponent of $p$ in the prime factorization of $s$ is $\alpha_j$ for some $j\in\{1,\dots,k\}$. Translates of such fibers in both directions can be added to construct more complicated examples.

\subsection{One scale, many primes}\label{571113}
In the example in Section \ref{example1}, we used two well separated scales ($p^{10}q^{10}$ and $pq$). If we allow $s_A$ to have 4 or more distinct prime factors, then $A$ may violate the assumptions of Proposition \ref{single prime} in other ways. 

The following example is due to Matthew Bond and the first author (unpublished).
Let $p_1,p_2,q_1,q_2$ be distinct primes. 
Assume that 
$p_1+q_1=p_2+q_2$, and, letting $N=p_1+q_1$, that $N$ is smaller than the product of any two distinct primes chosen from $\{p_1,p_2,q_1,q_2\}$. 
(For example, we could choose $p_1=5$, $p_2=7$, $q_1=13$, $q_2=11$, with $N=18$.) Let also $s_1=p_1q_1$ and $s_2=p_2q_2$.
Let $A$ be a set of $N$ integers such that 
\begin{itemize}
\item $A$ mod $s_1$ is a union of two $s_1$-fibers, one in each direction, of cardinalities $p_1$ and $q_1$,
\item $A$ mod $s_2$ is a union of two $s_2$-fibers, one in each direction, of cardinalities $p_2$ and $q_2$.
\end{itemize}
This is easily produced via the Chinese Remainder Theorem. 
Then $\Phi_{s_1}(X)$ and $\Phi_{s_2}(X)$ divide $A(x)$. It follows that $p_1p_2q_1q_2$ divides $s_A$. We further note that none of $\Phi_{p_1},\Phi_{p_2},\Phi_{q_1},\Phi_{q_2}$ divide $A(X)$.

Suppose that we have a factorization $s_A=s_{1,A}s_{2,A}$ satisfying the conditions of Proposition \ref{single prime}. If $s_{2,A}$ is divisible by at least two primes from  $\{p_1,p_2,q_1,q_2\}$, then $s_{2,A}>N=|A|$, violating the first condition of the proposition. Therefore at most one of our four primes may divide $s_{2,A}$. Then, however, at least three of them must divide $s_{1,A}$. It follows that $s_{1,A}$ is divisible by at least one of $s_1$ and $s_2$, violating the second condition. 

Nonetheless, it turns out that the construction in Section \ref{oneclustergamma} with the single cluster $\mathcal{C} = S_A = \{s_1, s_2 \}$ is sufficient in this case. Indeed, let $Q=q_1q_2$, so that $t_1=p_1$ and $t_2=p_2$. Then $T=\max(p_1,p_2)<|A|$, and Lemmas \ref{q-lemma1} and \ref{q-lemma2} provide the requisite SLV set.

\subsection{Lower bounds with more prime factors}\label{cuboid-lower-bound}

Our current methods are not sufficient to extend Theorem \ref{main theorem} to the case when more than two distinct prime factors are allowed. Below, we indicate a single-step cuboid argument leading to a lower bound on $|A|$ in certain situations. This is enough to resolve simple examples such as those below. However, there does not seem to be any easy way to iterate the argument to allow more complicated configurations of cyclotomic divisors. Our proof of Theorem \ref{main theorem} fails at multiple points in this setting.

\begin{lemma}\label{multi-step1}
Let $A\in\calm(\ZZ_M)$, where $M=\prod_{i=1}^K p_i^{n_i}$ and $p_1,\dots,p_K$ are distinct primes. 
Suppose that $m|(M/p_i)$ for some $i\in\{1,\dots,K\}$. 
Assume further that $\Phi_{mp_i}|A$. Then $p_i|\bbA^m[\Delta]$ for any $m$-cuboid $\Delta$.
\end{lemma}

\begin{proof}
We consider $A$ as a multiset in $\ZZ_N$, where $N=mp_i$. We define a family of weighted multisets $\Delta\in\calm(\ZZ_N)$, as follows. Let $\jfrak=\{j\in\{1,\dots,K\}:\ p_j|m\}$. If $p_i\nmid m$, we consider $\Delta$ of the form
\begin{equation}\label{multi-e4}
\Delta(X)=X^c\prod_{j\in\jfrak} (X^{d_j}-1),
\end{equation}
where $c\in\ZZ_N$ and $(d_j,N)=N/p_j$. If $p_i|m$, we instead consider
\begin{equation}\label{multi-e5}
\Delta(X)=X^c (X^{d_i}-1) \prod_{j\in\jfrak, j\neq i} (X^{d_j}-1),
\end{equation}
with $c,d_j$ as above for $j\neq i$, and with $(d_i,N)=N/p_i^2$. In both cases, the induced multiset $\Delta$ in $\ZZ_m$ is an $m$-cuboid, and any $m$-cuboid can be (non-uniquely) represented in this manner.

For $\nu=0,1,\dots,p_i-1$, define
$
\Delta_\nu (X)=X^{\nu N/p_i}\Delta(X).
$
We claim that 
\begin{equation}\label{combined-cuboids}
\bbA^m[\Delta]=\sum_{\nu=0}^{p_i-1} \bbA^N[\Delta_\nu].
\end{equation}
Indeed, we have
\begin{align*}
\sum_{\nu=0}^{p_i-1} \bbA^N[\Delta_\nu]
&= \sum_{\nu=0}^{p_i-1} \bbA^N[{\nu N/p_i}*\Delta_\nu]
\\
&= \sum_{\nu=0}^{p_i-1} \sum_{x\in\ZZ_N} w_A^N(x-{\nu N/p_i})w^N_\Delta(x)
\\
&= \sum_{x\in\ZZ_m} w_A^m(x)w^m_\Delta(x) = \bbA^m[\Delta],
\end{align*}
as claimed.

Assume first that $p_i\nmid m$. Then $\Delta_\nu(X)-\Delta_{\nu'}(X)$ are $N$-cuboids for $\nu\neq \nu'$. Since $\Phi_N|A$, it follows from Proposition \ref{cuboid} that the corresponding cuboid evaluations are 0, hence $\bbA^N[\Delta_\nu]$, $\nu=0,1,\dots,p_i-1$, are all equal. Thus $\bbA^m[\Delta]=p_i \bbA^N[\Delta_\nu]$
for any $\nu$.

If $p_i|m$, the argument is only slightly more complicated. For each $\nu$, we write $\Delta_\nu=\Delta^+_\nu-\Delta^-_\nu$, where
$$
\Delta^+_\nu(X)=X^{c+d_i+\nu N/p_i} \prod_{j\in\jfrak, j\neq i} (X^{d_j}-1),
\ \ 
\Delta^-_\nu(X)=X^{c+\nu N/p_i}  \prod_{j\in\jfrak, j\neq i} (X^{d_j}-1).
$$
Then $\Delta^+_\nu(X)-\Delta^+_{\nu'}(X)$ and $\Delta^-_\nu(X)-\Delta^-_{\nu'}(X)$ are $N$-cuboids for $\nu\neq \nu'$. Applying Proposition \ref{cuboid} as above, we see that $\bbA^N[\Delta^+_\nu]$, $\nu=0,1,\dots,p_i-1$, are all equal, and similarly for $\bbA^N[\Delta^-_\nu]$. Using this together with (\ref{combined-cuboids}), we get 
$$
\bbA^m[\Delta]=p_i \left(\bbA^N[\Delta^+_\nu]- \bbA^N[\Delta^-_\nu]\right)
$$
for any $\nu$, proving the lemma in this case.
\end{proof}

\begin{lemma}\label{scaling}
Let $M=\prod_{i=1}^K p_i^{n_i}$, where $p_1,\dots,p_K$ are distinct primes. 
Let $\Lambda:= \Lambda(c,p_i^\beta)$ for some $c\in\ZZ_M$, $1\leq i\leq K$, and $0<\beta<n_i$.
Suppose that $A\in\calm^+(\ZZ_M)$ is supported in $\Lambda$, in the sense that $w_A(x)=0$ for all $x\not\in\Lambda$. Let $M'=M/p_i^\beta$, and observe that the mapping
$$
\ZZ_{M'}\ni x \to c+ p_i^{\beta}x \in \Lambda
$$ is one-to-one. 
Define the ``rescaled'' multiset $A' \in\calm(\ZZ_{M'})$ by
\begin{equation}\label{e-scale}
w^{M'}_{A'}(x)=w_A(c+ p_i^{\beta}x).
\end{equation}
Then for any $m$ such that $p_i^\beta | m |M$, we have
$$
\Phi_m|A \ \ \Leftrightarrow \ \ \Phi_{m/p^\beta}|A'.
$$
\end{lemma}

\begin{proof}
Since the scaling in (\ref{e-scale}) maps $m$-cuboids supported in $\Lambda$ to $(m/p^\beta)$-cuboids in $\ZZ_{M'}$, the lemma follows from Proposition (\ref{cuboid}). 
\end{proof}

\begin{proposition}\label{multi-prop}
Let $A\in\calm(\ZZ_M)$, where $M=\prod_{i=1}^K p_i^{n_i}$ and $p_1,\dots,p_K$ are distinct primes. 
Suppose that there exist $i\in\{1,\dots,K\}$, $m_0|(M/p_i^{n_i})$, and $1\leq \alpha_1<\alpha_2<\dots<\alpha_\ell\leq n_i$ such that 
\begin{equation}\label{multi-e6}
\Phi_{m_1}\dots \Phi_{m_\ell}|A,\ \ \hbox{where }
m_j=m_0p_i^{\alpha_j},\ j=1,\dots,\ell.
\end{equation}
Then $p_i^\ell\mid \bbA^{m_0}[\Delta]$ for any $m_0$-cuboid $\Delta$.
\end{proposition}

\begin{proof}
The proof is by induction in $\ell$. In order to streamline the proof, we note that the statement of the proposition holds trivially for $\ell=0$, with no cyclotomic divisors assumed and the trivial conclusion $1=p_i^0| \bbA^{m_0}[\Delta]$. We will use this as the base case.

Assume now that $\ell\geq1$, and that the proposition is true in any cyclic group $\ZZ_{M'}$ with $\ell$ replaced by $\ell-1$.
The inductive step is similar to the proof of Lemma \ref{multi-step1}.
 Assume that (\ref{multi-e6}) holds.
Let $m=m_1/p_i$ and $N=m_1=mp_i$.
Define 
$$\jfrak:=\{j\in\{1,\dots,K\}:\ p_j|m_0\}.$$
For any $m_0$-cuboid $\Delta_0$, we may 
write $\bbA^{m_0}[\Delta_0]$ as a linear combination of expressions of the form $\bbA^{m}[\Delta]$, where
\begin{equation}\label{multi-e8}
\Delta(X)=X^c\prod_{j\in\jfrak} (X^{d_j}-1),
\end{equation}
with $c\in\ZZ_N$ and $(d_j,N)=N/p_j$ for each $j$. (The details are left to the interested reader, but we write out explicitly a very similar decomposition in (\ref{combined-cuboids}) in the proof of Lemma \ref{multi-step1}.)
We will prove that $p_i^\ell\mid \bbA^{m}[\Delta]$ for each such $\Delta$.

Define
$$
\Delta_\nu(X)=X^{c+\nu N/p_i}  \prod_{j\in\jfrak, j\neq i} (X^{d_j}-1)
$$
for $\nu=0,1,\dots,p_i-1$. Then
$$
\bbA^m[\Delta]= 
\sum_{\nu=0}^{p_i-1} \bbA^N[\Delta_\nu].
$$
Since $\Delta_\nu(X)-\Delta_{\nu'}(X)$ are $N$-cuboids for $\nu\neq \nu'$, Proposition \ref{cuboid} implies that $\bbA^N[\Delta_\nu]$ with $\nu=0,1,\dots,p_i-1$ are all equal. Thus
\begin{equation}\label{multi-e9}
\bbA^m[\Delta]= 
p_i \bbA^N[\Delta_\nu] \hbox{ for any }\nu.
\end{equation}
For each $\nu=0,1,\dots,p_i-1$, let $A_\nu:= A\cap\Lambda(c+\nu N/p_i, p_i^{\alpha_1})$. Let also $M':=M/p_i^{\alpha_1}$, and 
define the rescaled multisets $A'_\nu \in\calm(\ZZ_{M'})$ as in Lemma \ref{scaling}:
$$
w^{M'}_{A'_\nu}(x)=w_A(c+\nu N/p_i + p_i^{\alpha_1}x).
$$
By Lemmas \ref{grid-split} and \ref{scaling}, we have
$\Phi_{m'_2}\dots\Phi_{m'_\ell}|A'_\nu$ for each $\nu$, where 
$m'_j=m_0p_i^{\alpha_j-\alpha_1}$. Furthermore, let $\Delta'_\nu$ be the rescaling of $\Delta_\nu$, then $\Delta'_\nu$ is an $m_0$-cuboid in $\ZZ_{M'}$, and
 (with the obvious notation)
$$
\bbA^N[\Delta_\nu]= (\bbA'_\nu)^{m_0}[\Delta'_\nu].
$$
By the inductive assumption, the last quantity is divisible by $p_i^{\ell-1}$.
The conclusion follows by combining this with (\ref{multi-e9}).
\end{proof}

\begin{corollary}\label{multi-cor1}
Let $A\in\calm^+(\ZZ_M)$, where $M=\prod_{i=1}^K p_i^{n_i}$ and $p_1,\dots,p_K$ are distinct primes. Suppose that the assumptions of Proposition \ref{multi-prop} are satisfied, and, additionally, there exists $\alpha_0$ with
$0\leq\alpha_0< \alpha_1$ such that 
$$\Phi_{m_*}\nmid A,\ \hbox{ where } m_*=m_0p_i^{\alpha_0}.$$
Then $|A|\geq p_i^\ell$.
\end{corollary}

\begin{proof}
If $\alpha_0=0$, then there exists an $m_0$-cuboid $\Delta$ such that 
$\bbA^{m_0}[\Delta]\neq 0$. Combining this with Proposition \ref{multi-prop}, we see that
$$
|A|\geq |\bbA^{m_0}[\Delta]|\geq p_i^\ell,
$$
as claimed.

If $\alpha_0>0$, we find instead an $m_*$-cuboid $\Delta$ such that 
$\bbA^{m_*}[\Delta]\neq 0$. As in the proof of Lemma \ref{multi-step1}, write $\Delta(X)=\Delta^+(X) -\Delta^-(X)$, where $\Delta^+,\Delta^-$ are the faces of $\Delta$ perpendicular to the $p_i$ direction, so that each of $\Delta^+,\Delta^-$ must be contained in a single $p_i^{\alpha_0}$-grid. At least one of $\bbA^{m^*}[\Delta^+]$ and 
$\bbA^{m^*}[\Delta^-]$ must be nonzero. Assume without loss of generality that $\bbA^{m^*}[\Delta^+]\neq 0$, and let $\Lambda$ be the $p_i^{\alpha_0}$-grid containing $\Delta^+$. 
Consider the restriction $A\cap\Lambda$ of $A$ to $\Lambda$, and rescale it by a factor of $p_i^{-\alpha_0}$ as in Lemma \ref{scaling}. This reduces the proof to the case $\alpha_0=0$ as above.
The details are left to the reader.
\end{proof}

\noindent {\bf Example 6.1.}
Define $\calS_n$ for $n\in\NN$ as in (\ref{def-S}), with $A=B$.
Let $M:=lcm(S_A)=\prod_{i=1}^K p_i^{n_i}$, where $p_1,\dots,p_K$ are distinct primes. Let $m_0|(M/p_i^{n_i})$. Suppose that
$$
\{m_0p_i^{\alpha_1},m_0p_i^{\alpha_2},\dots,m_0p_i^{\alpha_\ell}\}\subset S_A
$$
for some $i\in\{1,\dots,K\}$ and $1\leq\alpha_1<\alpha_2<\dots<\alpha_\ell\leq n_i$. Assume furthermore that, in the notation of Section \ref{clusterdiv},
$$
\textsf{EXP} (i) : = \{\alpha_{1},\dots, \alpha_{\ell}\}.
$$
(In words, no other powers of $p_i$ appear in the prime factorization of elements of $S_A$.) 
By Corollary \ref{multi-cor1}, we have 
$|A|\geq p_i^\ell$. Since $|A|$ is relatively prime to all elements of $S_A$, the inequality must be strict. Thus the SLV construction in Section \ref{clusterdiv}, with the same choice of $i$, is sufficient in this case. It follows that $\calS_n$ satisfy the conclusions of Theorem \ref{Favard-two-primes} in this case.

\medskip

The corollary below extends Lemma \ref{multi-step1} in a different direction.

\begin{corollary}\label{multi-cor}
Let $A\in\calm^+(\ZZ_M)$, where $M=\prod_{i=1}^K p_i^{n_i}$ and $p_1,\dots,p_K$ are distinct primes. Suppose that $mp_1\dots p_I|M$ for some $I<K$. Assume further that
$\Phi_{p_im}|A$ for $i=1,\dots,I$, but $\Phi_m\nmid A$. Then 
$A(1)\geq p_1\dots p_I$.
\end{corollary}

\begin{proof}
By Lemma \ref{multi-step1}, we have $p_1\dots p_I|\bbA^m[\Delta]$ for every $m$-cuboid $\Delta$. Since 
$\Phi_m\nmid A$, by Proposition \ref{cuboid} there exists an $m$-cuboid $\Delta$ such that $\bbA^m[\Delta]\neq 0$. Hence for that $\Delta$, we have
$$
A(1)\geq |\bbA^m[\Delta]|\geq p_1\dots p_I.
$$
\end{proof}

\begin{remark}\label{too-much}
A stronger version of 
Corollary \ref{multi-cor} can be obtained, with the same proof, by using Proposition \ref{multi-prop} instead of Lemma \ref{multi-step1}. Since the statement would be significantly more complicated without contributing new ideas, we omit it here.
\end{remark}

\noindent {\bf Example 6.2.}
The following extends the example in Section \ref{571113}.
Let $A\in\calm^+(\ZZ_M)$, with $M=p_1\dots p_k Q$, where 
$p_1,\dots,p_k$ are distinct primes and $(p_1\dots p_k,Q)=1$. Suppose that 
$S_A=\{s_1,\dots,s_k\}$, where for each $s_j$ we have
$p_j|s_j|p_jQ$.

We claim that $|A|> \max(p_1,\dots,p_k)$. Indeed, let $j\in\{1,\dots,k\}$.
Since $\Phi_{s_j}|A$ and $\Phi_{s_j/p_j}\nmid A$, Corollary \ref{multi-cor} with $m=s_j/p_j$ implies that $|A|\geq p_j$. Since we have $(|A|,s_j)=1$ by the definition of $S_A$, the inequality must be strict.

For the purpose of an application to the Favard length problem, the single-cluster construction in Section \ref{oneclustergamma} works for such $A$.
Indeed, let $\mathcal{C}: = S_A$, and let $Q$ be as above, so that
$T=\max (p_1,\dots,p_k)$. In light of the upper bound above, this is sufficient.

\medskip

\noindent {\bf Example 6.3.}
Let $p,q,r$ be distinct primes. Suppose that $A\in\calm^+(\ZZ_M)$, where $pqr|M$, and that 
\begin{equation}\label{pqr-e1}
\{pq,pr,qr\}\subset S_A.
\end{equation}
By the definition of $S_A$, this implies that
\begin{equation}\label{pqr-e2}
(|A|, pqr)=1,
\end{equation}
and in particular none of $\Phi_p, \Phi_q,\Phi_r$ can divide $A$.

Let us try to apply the single-cluster construction in Section \ref{oneclustergamma} in this case.
With $\mathcal{C}: = S_A$, at most one of $p,q,r$ may divide $Q$, so that
$T$ must be divisible by at least two of them. Hence we need (at least) a bound of the form $|A|>\min(pq, pr, qr)$.

Corollary \ref{multi-cor} does indeed provide such a bound. Applying the corollary to $A$ with $m=q$, and using that $\Phi_{pq}$ and $\Phi_{qr}$ divide $A$ but $\Phi_q$ does not, we see that $|A|\geq qr$. By (\ref{pqr-e2}), the inequality must be strict. Interchanging the primes, we get that 
\begin{equation}\label{pqr-e3}
|A|>\max (pq, pr, qr).
\end{equation}
This is in fact sufficient, with e.g., $Q=q$ and $T=pr<|A|$.

The following example shows that the lower bound in (\ref{pqr-e3}) is essentially optimal up to a multiplicative constant.
Define $A\subset \ZZ_M$, where $pqr|M$, so that
$$
A=(a*F_p*F_q) \cup (a'*F_q*F_r) \cup (a''*F_p*F_r).
$$
(With $M/pqr$ sufficiently large, we can choose $a,a',a''$ so that they belong to different $D(M)$-grids. Then the three ``components" above are disjoint.) 
Then (\ref{pqr-e1}) holds, and 
$|A|=pq+pr+qr$, matching the order of magnitude of (\ref{pqr-e3}) if $p,q,r$ have about the same size.

We note, however, that the argument does not extend to situations when two or more elements of $S_A$ do not have a common ``direct parent''. For example, Corollary \ref{multi-cor} does not provide any improvement on the single-step bound from Lemma \ref{multi-step1} if $S_A=\{p^2q^3, p^4r^5, q^6 r^7\}$, nor does the stronger version of it mentioned in Remark \ref{too-much}.

\section{Proof of Theorem \ref{favard-small}}\label{small-card}

We will rely on the existing results on vanishing sums of roots of unity as in (\ref{vanishing}) with small $k$ \cite{PR}, \cite{CDK}. A vanishing sum of the form (\ref{vanishing}) is called {\em minimal} if there is no proper subset $\{i_1,\dots,i_\ell\}\subsetneqq\{1,\dots,k\}$ such that $z_{i_1}+\dots+z_{i_\ell}=0$. 
For small $k$, all minimal vanishing sums of roots of unity with $k$ elements can be classified and enumerated explicitly. 
Such an enumeration is provided in \cite[Table 1]{PR} for $k\leq 12$, and extended in  \cite{CDK} to $k\leq 16$.

The results of \cite{PR}, \cite{CDK} can be converted to our language of arrays and fibers from Section \ref{cyclo-tools} as follows. Recall from Section 1 that 
any vanishing sum of roots of unity can be phrased in terms of cyclotomic divisibility of polynomials. Conversely,  if $A(X)$
is the mask polynomial corresponding to some $A\subset\NN$ and $\Phi_N(X)|A(X)$ for some $N \in\NN\setminus\{1\}$, then we must have
$$A(e^{2\pi i /N})=\sum_{a\in A}e^{2\pi i a/N}=0.
$$
In other words, $A(e^{2 \pi i /N})$ forms a vanishing sum of roots of unity. We may further reduce modulo $N$, so that $A$ mod $N$ is a multiset in $\calm^+(\ZZ_N)$. For each fixed $N$, the correspondence between multisets $A$ mod $N$ in $\calm^+(\ZZ_N)$ satisfying $\Phi_N(X)|A(X)$ and vanishing sums of roots of unity of the form
\begin{equation}\label{small-e2}
\sum_{a\in\ZZ_N} w(a) e^{2\pi i a/N}=0
\end{equation}
is one-to-one. The above sum is minimal if and only if $A$ has the following minimality property.

\smallskip\noindent
{\bf Property (M).} There is no multiset $A'\in\calm^+(\ZZ_N)$ such that $A\neq A'$, $\Phi_N(X)|A'(X)$, and $w_{A'}(x)\leq w_A(x)$ for all $x\in\ZZ_N$. 
\smallskip

If the sum is not minimal, its decomposition into minimal vanishing sums corresponds to writing 
$A(X)$ as a sum of polynomials $A_j(X)$ which have that property.  This decomposition into minimal relations is the focus of \cite{PR} and \cite{CDK}; that such a decomposition exists can be proven using induction upon the weight functions $w$ appearing in the vanishing sums \eqref{small-e2}.

 By the equivalence (i) $\Leftrightarrow$ (iii) in Proposition \ref{cuboid}, all polynomials $A(X)$ which do have the minimality property must fall into one of the following categories:
\begin{itemize}
\item $A(X)$ is an $N$-fiber in some direction.
\item $A(X)$ is an ``irreducible" linear combination of $N$-fibers as in Proposition \ref{cuboid} (iii) that cannot be expressed as a linear combination of $N$-fibers with nonnegative coefficients.
\end{itemize}

The first category of minimal sums is referred to in \cite{PR}, \cite{CDK} as $R_p$, where the prime $p$ indicates the direction of the $N$-fiber. Such configurations can only occur when $p|N$. The second category is described using a language of recursive relations. For sets of cardinality at most 10, we will only need the $(R_p:kR_q)$ notation of \cite{PR}. Here,  the integers $p,q$ are distinct primes and we always have $1\leq k<p$. 

In our language, a given configuration of type $(R_p : k R_q)$ is constructed as follows. We choose some $N$ with $2 pq | N$ and work in $\ZZ_N$. Start with an $N$-fiber $x*F^N_p$ in the $p$ direction. Choose $k$ points $x_1,\dots,x_k$ of that fiber, and subtract $x_j*F^N_q$ for each $j$. This ``cancels" the points $x_1,\dots,x_k$, which now have weights 0, and introduces $k(q-1)$ points with negative weight $-1$. Finally, add a $N$-fiber in the 2 direction through each point with weight $-1$. This ``cancels" all the negative weights, and introduces additional $k(q-1)$ points with weight 1. The total weight of the configuration, modulo $\ZZ_N$, is $(p-k)+k(q-1)=p+kq-2k$.

\begin{example}\label{7example}
An example of a configuration of type $(R_5:2R_3)$ is provided by a multiset $\Xi\in\calm^+(\ZZ_N)$ with the mask polynomial
\begin{equation}\label{small-e1}
\begin{split}
\Xi(X) &= F^N_5(X)- (X^{N/5}+X^{2N/5})F^N_3(X) + \left(\sum_{i,j=1}^2 X^{\frac{iN}{5}+\frac{jN}{3}} \right)F^N_2(X)
\\
&= 1+X^{3N/5}+X^{4N/5}+ \sum_{i,j=1}^2 X^{\frac{iN}{5}+\frac{jN}{3}+\frac{N}{2}},
\end{split}
\end{equation}
where $N$ is divisible by $2\cdot 5\cdot 3=30$. We have $\Xi(1)=5+2\cdot 3-2\cdot 2=7$. Note that there are multiple configurations of the same type, depending on the placement of fibers. We illustrate one presentation below, which highlights the fiber geometry underpinning equation \eqref{small-e1}. The drawings use the coordinate representation introduced at the beginning of Section \ref{grids-fibers}.

\begin{figure}[H]
\centering
\includegraphics[scale=.1]{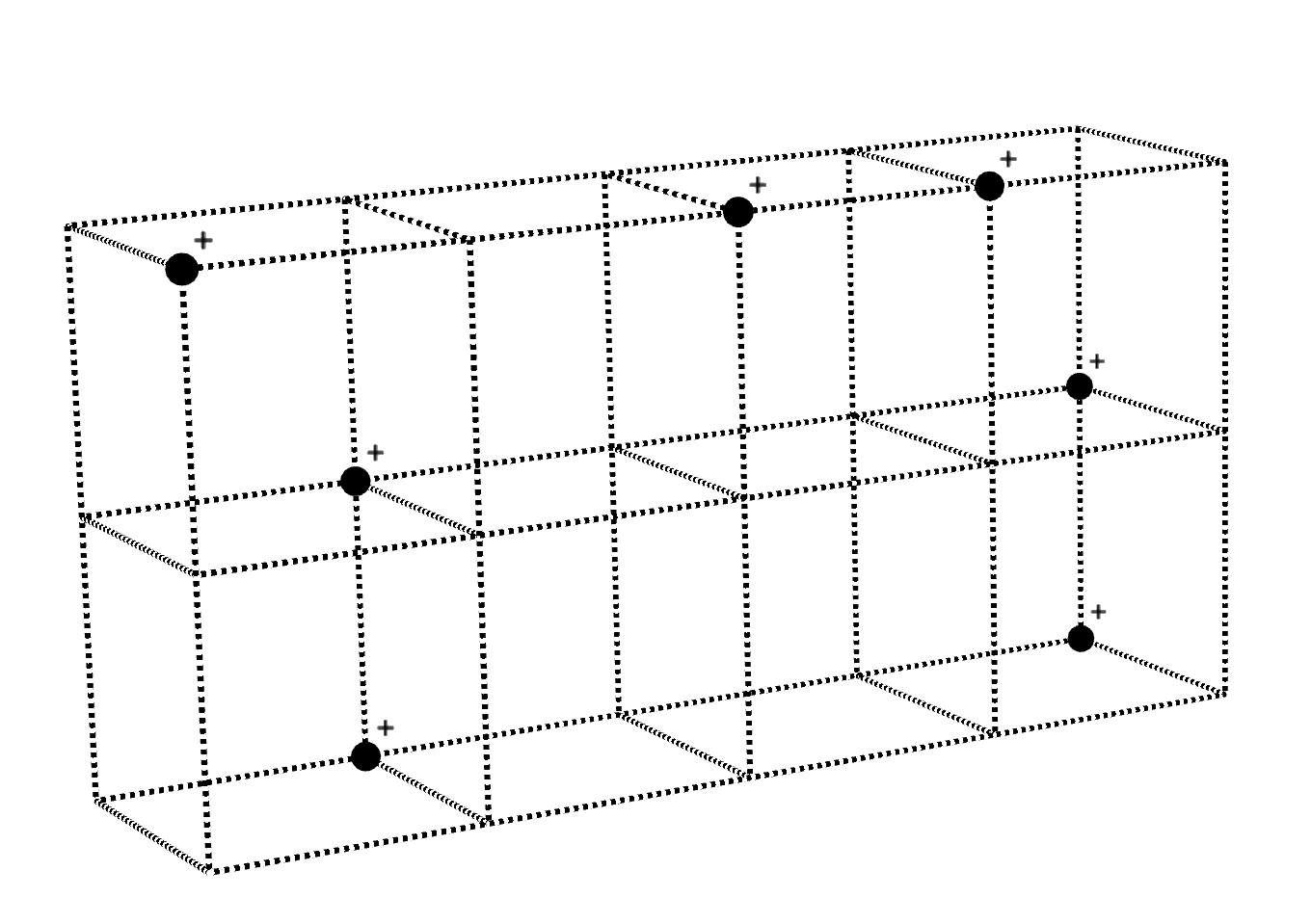}
\caption{We construct the above multiset $\Xi \in \mathcal{M}^+ (\mathbb{Z}_{30})$ of cardinality $7$ by combining $2$-fibers, $3$-fibers and $5$-fibers with appropriately chosen $\pm 1$ weights.}
\end{figure}

\begin{figure}[H]
\centering
\includegraphics[scale=.1]{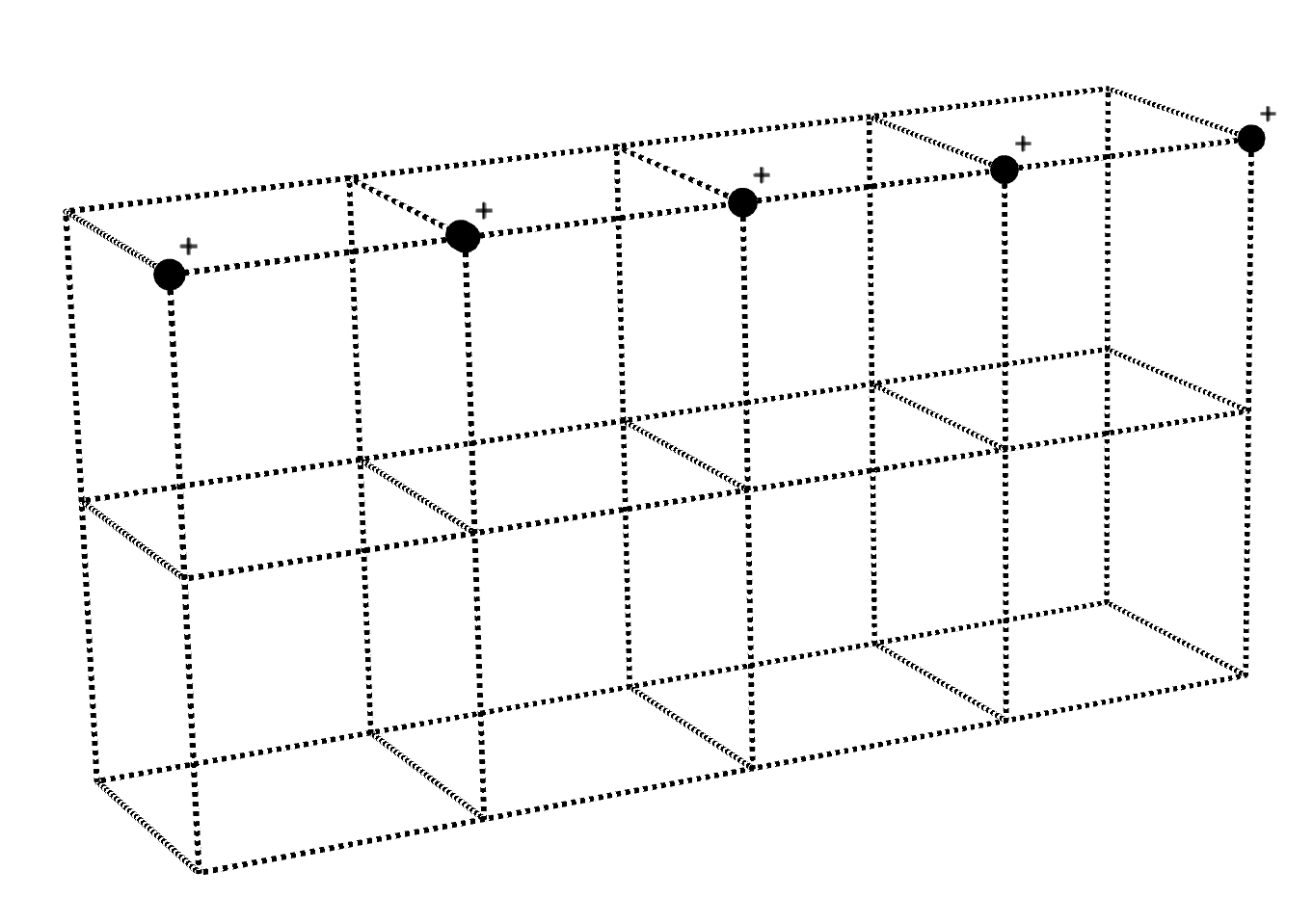}
\caption{Begin with any $30$-fiber in the $p = 5$  direction. In this example, we chose $x *F_5^{30} = (0,0,2) * F_5^{30}$. Assign weight $+1$ to this fiber.}
\end{figure}

\begin{figure}[H]
\centering
\includegraphics[scale=.1]{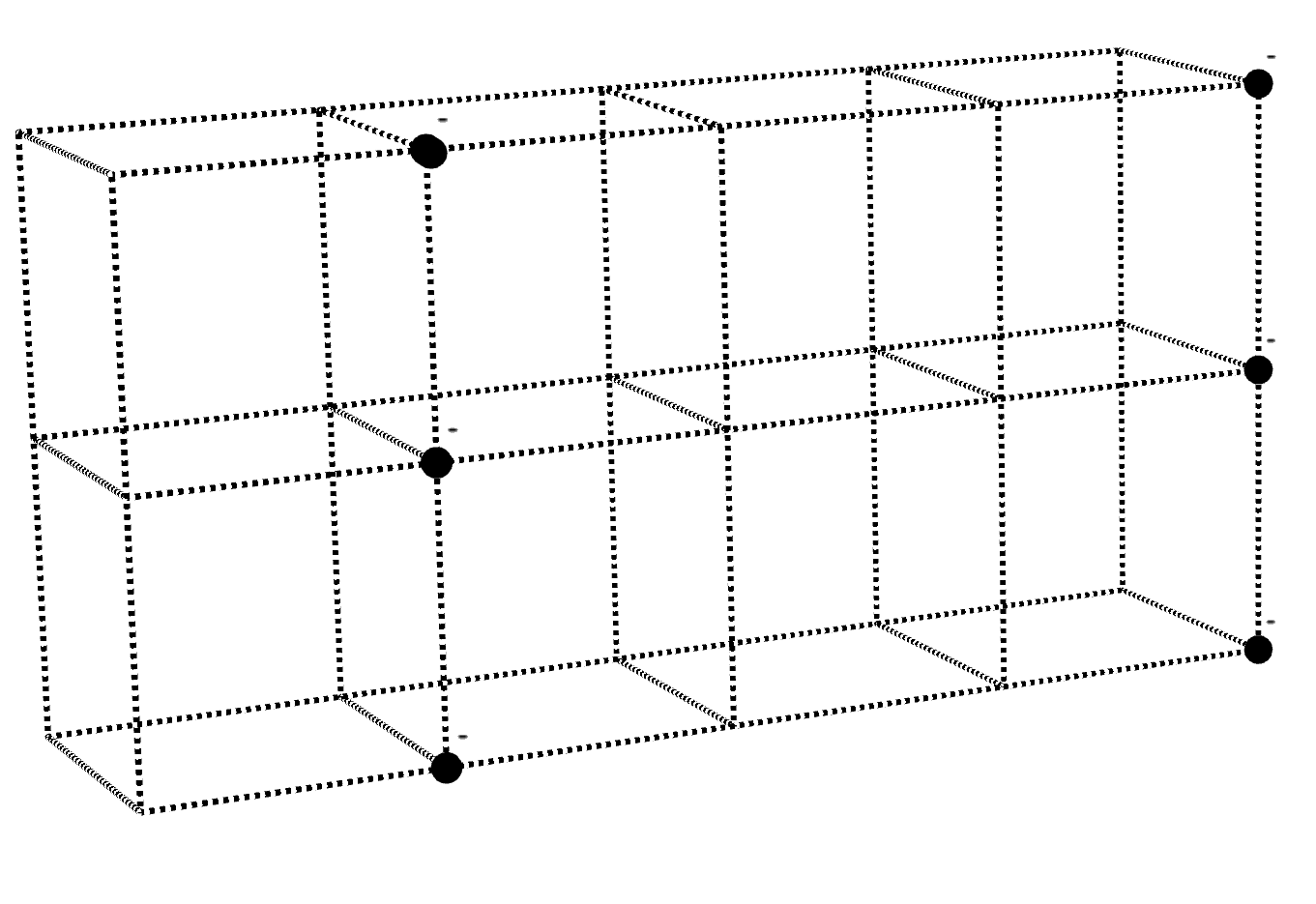}
\caption{Select points $x_{j_1}, x_{j_2}$ from the previous fiber and subtract $x_{j_1} * F_{3}^{30}$ and $x_{j_2} * F_{3}^{30}$. In this example, we chose $x_{j_1} = (1,0,2)$ and $x_{j_2} = (4,0,2)$.  This introduces $k(q - 1) = 2 \times 2 = 4$ points of weight $-1$ to our initial set.}
\end{figure}

\begin{figure}[H]
\centering
\includegraphics[scale=.1]{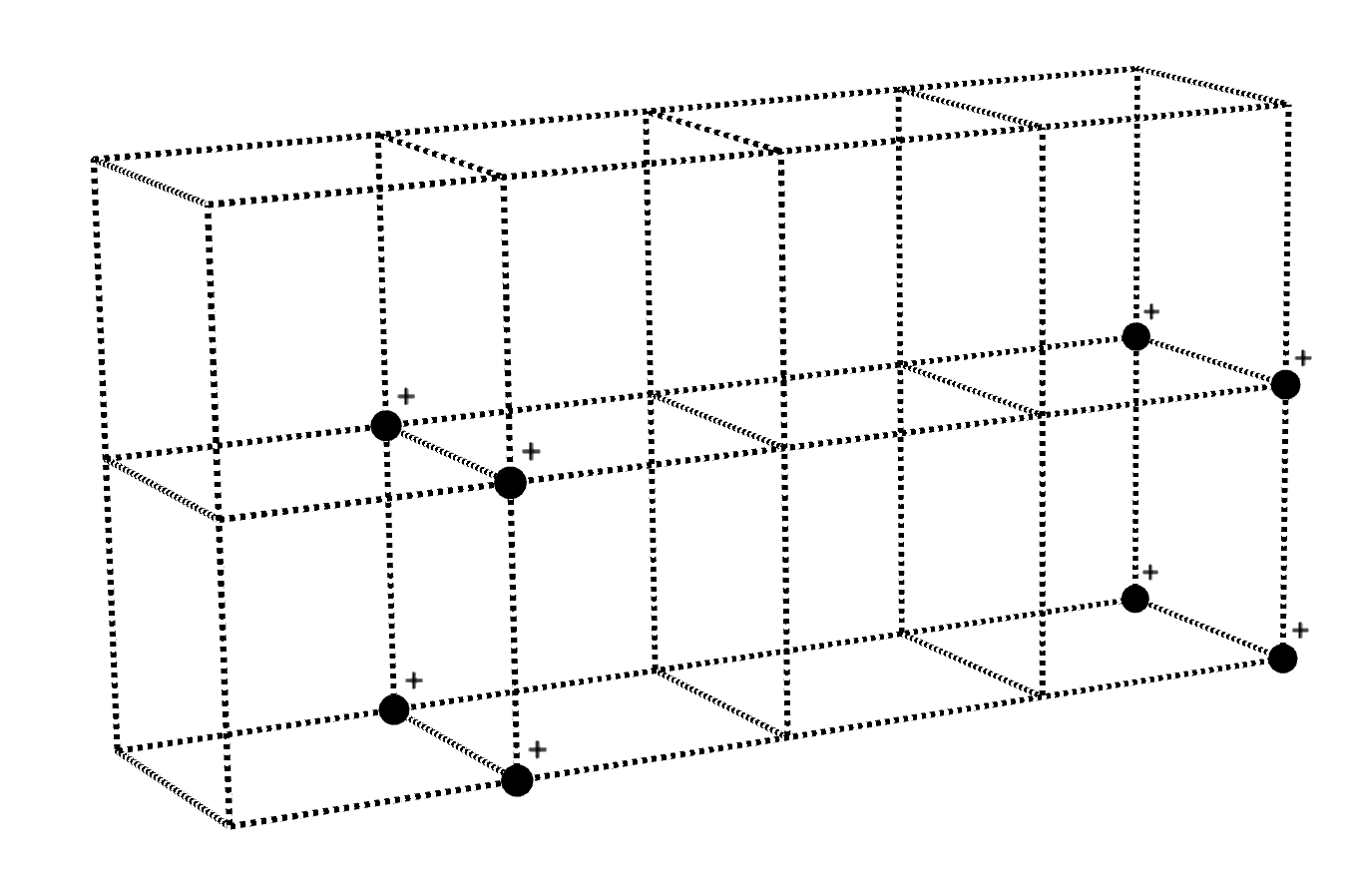}
\caption{We cancel the remaining points of the fibers $x_{j_1} * F_q^{N}$ and $x_{j_2}*F_q^{N}$ by adding $k(q - 1)$ $30$-fibers in the $2$ direction with weight $+1$. After cancellation, this adds $k(q-1)$ points of weight $+1$ to our set.}
\end{figure}

\end{example}

We now turn to sets of small cardinality. 
By Proposition \ref{lb-general-gamma}
and the proof of  Theorem \ref{Favard-two-primes} in Section \ref{sec2}, it suffices to prove the following.

\begin{lemma}\label{small-lemma} 
Let $A\subset\NN$. Assume that $|A|\leq 10$, and that $S_A\neq \emptyset$.
Then there exists a prime $p_1$, relatively prime to $|A|$, such that:
\begin{itemize}
\item[(i)] $p_1|s$ for all $s\in S_A$,

\item[(ii)] in the notation of Proposition \ref{lb-general-gamma}, we have $p_1^{E_1}<|A|$.
\end{itemize}
\end{lemma}

Lemma \ref{small-lemma} replaces Theorem \ref{main theorem} in the proof of  Theorem \ref{Favard-two-primes}. In this case, the cardinality of $|A|$ is given, and we just need to find an appropriate prime $p_1$ to use in Proposition \ref{lb-general-gamma}. The lemma can fail for larger sets: for instance, Section \ref{571113} provides a counterexample with $|A|=18$.

\begin{proof}[Proof of Lemma \ref{small-lemma}.]
Let $A$ satisfy the assumptions of the lemma. 
Let $N\in S_A$ (note that $N$ divides $s_A$ but does not have to be equal to it). 
We write $A(X)$ mod $X^N-1$ as a sum of mask polynomials $A_j(X)$ satisfying the minimality condition (M) in $\ZZ_N$, each of which must be a fiber or an irreducible structure (as discussed previously). Since $N\in S_A$, we have $(N,|A|)=1$. Therefore:
\begin{itemize}
\item If $A_j$ is an $N$-fiber in the $p$ direction for some $p|N$, then $(p,|A|)=1$.
\item If $A_j$ is an irreducible structure of type $(R_p:kR_q)$ as described above, we must have $2pq|N$, hence $(2pq,|A|)=1$.
\end{itemize}
Considering all the minimal structures listed in \cite{PR}, and applying the above constraints, we are left with 
the following cases:
\begin{itemize}
\item $|A|=5$, and $A$ is a union of one $N$-fiber in the 2 direction and one $N$-fiber in the 3 direction,
\item $|A|=7$, and $A$ is one of the following:
\begin{itemize}
\item[(a)] a union of one $N$-fiber in the 2 direction and one $N$-fiber in the 5 direction,
\item[(b)] a union of two $N$-fibers in the 2 direction and one $N$-fiber in the 3 direction,
\item[(c)] an irreducible structure of type 
$(R_5:2R_3)$ (see Example \ref{7example}),
\end{itemize}

\item $|A|=8$, and $A$ is a union of one $N$-fiber in the 3 direction and one $N$-fiber in the 5 direction,
\item $|A|=9$, and $A$ is a union of two $N$-fibers in the 2 direction and one $N$-fiber in the 5 direction,
\item $|A|=10$, and $A$ is a union of one $N$-fiber in the 3 direction and one $N$-fiber in the 7 direction,

\end{itemize}

We first consider the case when $A$ is a disjoint union of $N$-fibers in 2 different directions. Let $p_1=3$ if $|A|=5$, $p_1=5$ if $|A| \in\{8,9\}$, and $p_1=7$ if $|A|=10$. We claim $E_1=1$, so that the conclusion of the lemma holds in each case. 

Let $\alpha$ be the exponent of $p_1$ in the prime factorization of $N$; since $A$ contains an $N$-fiber in the $p_1$ direction, we must have $\alpha\geq 1$. Suppose that $s\in S_A$, $s\neq N$, and that the exponent of $p_1$ in the prime factorization of $s$ is $\beta$. If $\beta>\alpha$, then $A$ mod $s$ has nonempty intersection with at least $p_1$ different $D(s)$-grids in $\ZZ_s$. By Lemma \ref{grid-split}, we must have $\Phi_s|A\cap \Lambda$ for each such grid, hence $A$ mod $s$ can be written as a union of at least $p_1$ non-empty multisets $A_j\in\calm^+(\ZZ_s)$ with $\Phi_s|A_j$ for all $j$. But that is not compatible with the above list of permitted structures. If $\beta<\alpha$, the same argument applies, but with $N$ and $s$ interchanged.

Assume now that $|A|=7$, and let $p_1=2$. For each $s\in S_A$, $A$ mod $s$ has one of the structures listed in (a)-(c) above, hence $2|s$. To complete the proof of the lemma, we need to prove that $E_1\leq 2$. We start with an auxiliary result.

\medskip\noindent
{\bf Claim.} Suppose that $B\subset \NN$ satisfies $|B|\leq 4$, and that $\Phi_{s_1}(X)\Phi_{s_2}(X)|B(X)$ for some $s_1,s_2\in\NN\setminus\{1\}$ such that the exponents $\alpha_1,\alpha_2$ of 2 in the prime factorization of $p_1$ and $p_2$ satisfy $1\leq\alpha_1<\alpha_2$. Then one of the following holds:
\begin{itemize}
\item[(i)] $|B|=3$, and $B$ mod $s_i$ is a translate of $F^{s_i}_3$ for both $i=1$ and $i=2$,
\item[(ii)] $|B|=4$, and $B$ mod $s_2$ is a translate of a ``double fiber" in the 2 direction, with 
\begin{equation}\label{small-e3}
B(X)=X^b(1+X^{s_2/2})(1+X^{c}) \mod X^{s_2}-1
\end{equation}
for some $b,c\in \ZZ_{s_A}$ such that $c\equiv s_1/2$ mod $s_1$.
\end{itemize}

\smallskip\noindent
\begin{proof}[Proof of Claim.] Since the smallest irreducible structure has 5 elements, $B$ mod each $s_i$ must be a union of $s_i$-fibers in the directions of 2 or 3. If $B$ mod $s_1$ contains a translate of $F^{s_1}_3$, then $|B|=3$ and we are in case (i). Suppose now that $B$ mod $s_1$ is a union of (one or two) 2-fibers. Then $B$ mod $s_1$ has nonempty intersection with at least two $2^{\alpha_1}$-grids in $\ZZ_{s_1}$. It follows that $
B$ mod $s_2$ has nonempty intersection with at least two $2^{\alpha_1}$-grids $\Lambda_1,\Lambda_2$ in $\ZZ_{s_2}$, and by Lemma \ref{grid-split}, we must have $\Phi_{s_2}|B\cap \Lambda_i$ for $i=1,2$. Thus $
B\cap \Lambda_i$ must be an $s_2$-fiber in the 2 direction for each $i\in\{1,2\}$. This implies the conclusion of case (ii).
\end{proof}

We now return to the proof of Lemma \ref{small-lemma} in the case $|A|=7$, with $p_1=2$. Assume for contradiction that $s_i\in S_A$ for $i=1,2,3$, with $1\leq\alpha_1<\alpha_2<\alpha_3$, where $\alpha_i$ is the exponent of 2 in the prime factorization of $A$. 

First, $A$ mod $s_1$ must have one of the forms listed in (a)-(c) above with $N=s_1$. In each case, it follows that $A$ has nonempty intersection with at least two $2^{\alpha_1}$-grids in $\ZZ_{s_1}$, therefore also in $\ZZ_{s_2}$. For each such grid $\Lambda_j$, $A\cap\Lambda_j$ must be divisible by $\Phi_{s_2}\Phi_{s_3}$, and at least one such set must have cardinality at most 3. Applying the claim above, we see that there can be at most two such grids $\Lambda_1$ and $\Lambda_2$, and that (after possibly relabelling the grids) $A_i:=A\cap\Lambda_i$ with $i=1,2$ satisfy, respectively, the conclusions (i) and (ii) of the claim.

We now return to $\ZZ_{s_1}$, where $A_1$ and $A_2$ are multisets with $|A_1|=3$ and $|A_2|=4$ contained in disjoint $2^{\alpha_1}$-grids. Considering the structures in (a)-(c), we see that only (c) is compatible with this.

Let $\gamma_i$ be the exponent of 3 in the prime factorization of $s_i$ for $i=1,2,3$. Since $A_1$ must be a fiber in the 3 direction in both $\ZZ_{s_2}$ and $\ZZ_{s_3}$, we have $\gamma_2=\gamma_3$. Since the points of $A_1$ form a part of an $s_1$-fiber in the 5 direction in $\ZZ_{s_1}$, we must have $\gamma_1<\gamma_2$. Then, however, $A_2$ is contained in a single $3^{\gamma_1}$-grid in $\ZZ_{s_1}$, so that $A$ cannot form an
$(R_5:2R_3)$ structure there. This exhausts all possible cases, and ends the proof of the lemma.
\end{proof}

As mentioned in the introduction, it is likely that the cardinality bound $|A|\leq 10$ could be improved further with additional work along similar lines. However, for $|A|\geq 11$ the number of cases to consider increases rapidly, making the task significantly more time-consuming.



\section*{Acknowledgments} 
We would like to thank the anonymous referee for valuable comments.

\bibliographystyle{amsplain}

\begin{thebibliography}{99}
\bibitem{BV} M. Bateman, A. Volberg, {\it An estimate from below for the Buffon 
needle probability of the four-corner Cantor set}, {\it Math.\ Res.\ Lett.} {17} (2010), 959-967.


\bibitem{BThesis}
M. Bond, {\it Combinatorial and Fourier Analytic $L^2$ Methods For Buffon's Needle Problem}, Ph, D, thesis, University of Michigan, http://bondmatt.wordpress.com/2011/03/02/thesis-second-complete-draft/.

\bibitem{BLV} M.\ Bond, I.\ {\L}aba, and A.\ Volberg, Buffon needle estimates for rational product Cantor sets, Amer. J. Math. 136 (2014), 357-391.


\bibitem{BV1}
M. Bond, A. Volberg: {\it Buffon needle lands in $\epsilon$-neighborhood of a 1-dimensional 
Sierpinski Gasket with probability at most $|\log\epsilon|^{-c}$}, Comptes Rendus Mathematique, Volume 348, Issues 11-12, June 2010, 653--656.

\bibitem {BV2}
M. Bond, A. Volberg: {\it Circular Favard Length of the Four-Corner Cantor Set}, J. of Geometric Analysis 21 (2011), 40--55.

\bibitem{BV3} 
M. Bond, A. Volberg, {\it Buffon's needle landing near Besicovitch irregular self-similar sets},
Indiana Univ. Math. J. 61 (2012), no. 6, 2085--2109.

\bibitem{Bongers} T. Bongers, {\it Geometric bounds for Favard length}, Proc. Amer. Math. Soc. {147}, (2019)

\bibitem{BT} T. Bongers, K. Taylor, {\it Transversal families of nonlinear projections and generalizations of Favard length}, available as preprint at arXiv:2105.01708  

\bibitem{CDOV} A. Chang, D. D\c{a}browski, T. Orponen, M. Villa, \textit{Structure of sets with nearly maximal Favard length}, preprint, arXiv:2203.01279

\bibitem{CDK} L. Christie, K. Dykema, I. Klep, {\it Classifying minimal vanishing sums of roots of unity}, preprint, arXiv:2008.11268

\bibitem{CDT} L. Cladek, B. Davey, K. Taylor, \textit{Upper and lower bounds on the rate of decay of the Favard curve length for the four-corner Cantor set}, to appear in Indiana U. Math. J. (2022).

\bibitem{CJ} J.H. Conway, A.J. Jones: {\it Trigonometric diophantine equations (On vanishing sums of roots of unity)}, Acta Arithmetica 30 (1976), 229--240.

\bibitem{CS} D. Coppersmith, J.P. Steinberger: {\it On the entry sum of cyclotomic arrays}, Integers: the Electronic Journal of Combinatorial and Additive Number Theory, 6 (2006), \# A26.


\bibitem{DV} D. D\c{a}browski, M. Villa, {\it Analytic capacity and dimension of setx with plenty of big projections}, preprint, 2022, arXiv:2204:05804


\bibitem{deB} N.G. de Bruijn, {\it On the factorization of cyclic groups}, Indag. Math. 15 (1953), 370--377.



\bibitem{KMSV} G. Kiss, R. D. Malikiosis, G. Somlai, M. Vizer, 
{\it On the discrete Fuglede and Pompeiu problems}, Analysis \& PDE 13 (2020), 765-788.

\bibitem{L1} I. {\L}aba, {\it Recent progress on Favard length estimates for planar Cantor sets}, in: Operator-Related Function Theory and Time-Frequency Analysis, Proceedings of the 2012 Abel Symposium, K. Grochenig, Y. Lyubarskii, K. Seip, eds., Springer 2015, pp. 117-145.


\bibitem{LaLo1} I. {\L}aba, I. Londner, {\it Combinatorial and harmonic-analytic methods for integer tilings},
Forum of Mathematics - Pi 10:e8 (2022), 1--46.

\bibitem{LZ} I. {\L}aba, K. Zhai: {\it The Favard length of product Cantor sets},
Bull. London Math. Soc. 42 (2010), 997--1009.

%

\bibitem{LL} T.Y. Lam and K.H. Leung, {\it On vanishing sums of roots of unity},
J. Algebra 224 (2000), 91--109.

\bibitem{Mann} H. B. Mann, {\it On Linear Relations Between Roots of Unity}, Mathematika 12, Issue 2 (1965), 107--117.

\bibitem{Mattila} P. Mattila, {\it Orthogonal projections, Riesz capacities, and 
Minkowski content}, Indiana Univ. Math. J. 124 (1990), 185--198.

\bibitem{mattila}
P. Mattila: {\it Geometry of Sets and Measures in Euclidean Spaces}, Cambridge University Press, 1995.

\bibitem{NPV} F. Nazarov, Y. Peres, A. Volberg: {\em  The power law for the Buffon needle probability of the four-corner Cantor set}, 
Algebra i Analiz 22 (2010), 82--97; translation in St. Petersburg Math. J. 22 (2011),  61--72.

\bibitem{PS1} Y. Peres, B. Solomyak, {\it How likely is buffon's needle to fall near a
planar Cantor set?}, Pacific J. Math. {\bf 24} (2002), 473--496.

\bibitem{PR}
B. Poonen and M. Rubinstein: {\it Number of Intersection Points Made by the Diagonals of a Regular Polygon}, SIAM J. Disc. Math. 11 (1998), 135--156.

\bibitem{Re1} L. R\'edei, {\it \"Uber das Kreisteilungspolynom}, Acta Math. Hungar. 5 (1954), 27--28.

\bibitem{Re2} L. R\'edei, {\it Nat\"urliche Basen des Kreisteilungsk\"orpers}, Abh. Math. Sem. Univ. Hamburg 23 (1959), 180--200.

\bibitem{schoen}
I. J. Schoenberg, {\it A note on the cyclotomic polynomial}, Mathematika 11 (1964), 131-136.

\bibitem{Steinberger} J. P. Steinberger, {\it Minimal vanishing sums of roots of unity with large coefficients},
Proc. London Math. Soc. 97 (3) (2008), 689--717.

\bibitem{Tao} T. Tao, {\it A quantitative version of the Besicovitch projection
theorem via multiscale analysis},  Proc. London Math. Soc. 98 (2009), 559--584.

\bibitem{To} X. Tolsa, {\it Analytic capacity, rectifiability, and the Cauchy integral}, Proceedings of the ICM 2006, Madrid.

\bibitem{VV} D. Vardakis, A. Volberg, {\it The Buffon's needle problem for random planar disk-like Cantor sets}, preprint, 2022, arXiv:2205.14559

\bibitem{Wilson} B. Wilson, {\it Sets with arbitrarily slow Favard decay}, preprint available at arxiv.org/abs/1707.08137

\bibitem{Zhang} S. Zhang, {\it The exact power law for Buffon's needle landing near some random Cantor sets}, Rev. Mat. Iberoamericana 36(2) (2019), 537--548.
 
\end{thebibliography}


\begin{dajauthors}
\begin{authorinfo}[ilaba]
  Izabella {\L}aba\\
  University of British Columbia\\
  Vancouver, Canada\\
  ilaba{\@}math.ubc.ca \\
  \url{https://personal.math.ubc.ca/~ilaba/}
\end{authorinfo}
\begin{authorinfo}[cmarshall]
  Caleb Marshall\\
  University of British Columbia\\
  Vancouver, Canada\\
  cmarshall{\@}math.ubc.ca \\
  \url{https://www.math.ubc.ca/user/2587}
\end{authorinfo}
\end{dajauthors}

\end{document}